\documentclass{amsart}

\usepackage{amsmath,amssymb,amsthm,enumitem}
\usepackage[colorlinks=true,linkcolor=blue,citecolor=magenta]{hyperref}

\newcommand{\x}{\times}
\newcommand{\<}{\langle}
\renewcommand{\>}{\rangle}

\renewcommand{\a}{\alpha}
\renewcommand{\b}{\beta}
\renewcommand{\d}{\delta}

\newcommand{\e}{\epsilon}

\newcommand{\n}{\nu}
\newcommand{\s}{\sigma}

\renewcommand{\S}{\Sigma}
\renewcommand{\t}{\tau}
\newcommand{\ph}{\varphi}
\newcommand{\z}{\zeta}

\renewcommand{\i}{\infty}

\newcommand{\cG}{{\mathcal G}}

\newcommand{\cK}{{\mathcal K}}
\newcommand{\cL}{{\mathcal L}}

\newcommand{\cT}{{\mathcal T}}

\newcommand{\cW}{{\mathcal W}}

\newcommand{\bZ}{{\mathbf Z}}
\newcommand{\bR}{{\mathbf R}}

\newcommand{\bN}{{\mathbf N}}

\newcommand{\bS}{{\mathbb S}}
\newcommand{\bB}{{\mathbb B}}

\newtheorem{lemma}{Lemma}[section]
\newtheorem{lem}{Lemma}[section]
\newtheorem{prop}{Proposition}[section]
\newtheorem{cor}{Corollary}[section]

\DeclareMathOperator{\id}{id}
\DeclareMathOperator{\Supp}{Supp}
\DeclareMathOperator{\interior}{int}

\newcommand{\cl}{\overline}
\newcommand{\lv}{\left \lvert}
\newcommand{\rv}{\right \rvert}

\newcommand{\PL}{\mathsf{PL}_+}

\newcommand\restr[2]
  {#1 \upharpoonright #2}

\vfuzz \hfuzz
\title[]
{On some algebraic and analytic properties of the finitely generated simple left orderable groups $G_\rho$.}

\author{Pawel Aleksander Fedorynski} 
\address{Department of Mathematics, University of Hawaii at Manoa}\email{pfedor@gmail.com}
\author{Yash Lodha}
\address{Department of Mathematics, University of Hawaii at Manoa}\email{lodha@hawaii.edu}

\begin{document}
\newtheorem{theorem}{Theorem}[section]
\newtheorem{proposition}{Proposition}[section]
\theoremstyle{definition} \newtheorem{definition}{Definition}[section]

\dedicatory{} 

\begin{abstract}
In $2019$ Hyde and the second author constructed the first family of finitely generated, simple, left-orderable groups. We prove that these groups are not finitely presentable, non-inner amenable,
don't have Kazhdan's property (T) (yet have property FA), and that their first $\ell^2$-Betti number vanishes.
We also show that these groups are uniformly simple, providing examples of uniformly simple finitely generated left-orderable groups. Finally, we also describe the structure of the groups $G_{\rho}$ where $\rho$ is a periodic labelling.
\end{abstract}

\maketitle

\section{Introduction}

Whether a countable group admits a faithful action by orientation preserving homeomorphisms on the real line admits a surprisingly elementary algebraic characterization. Such an action exists if and only if the group is left-orderable: it admits a total order which is invariant under left multiplication by group elements.
The theory of orderable groups has a rich history that goes back to the work of Dedekind and H\"older \cite{GOD}, and has witnessed considerable recent advances \cite{NavasICM}. 

It is natural to inquire whether there exist finitely generated simple left-orderable groups. This was posed by Rhemtulla in $1980$ \cite{HydeLodha}, and was open for nearly four decades until it was answered in the affirmative by Hyde and the second author in \cite{HydeLodha}. Since then, several families have emerged \cite{HydeLodhaRivas}, \cite{MatteBonTriestino}, and \cite{HydeLodhaFP}. 

The construction in \cite{HydeLodha} takes as input a combinatorial map $\rho:\mathbf{Z}[\frac{1}{2}]\to \{a,a^{-1},b,b^{-1}\}$, called a quasi-periodic labeling, that satisfies a set of axioms. Using this input, the construction provides a finitely generated simple group $G_{\rho}\leq \textup{Homeo}^+(\mathbf{R})$ (see Section \ref{Prel} for more details). The purpose of this note is to establish some key structural properties of these groups. 

A key starting point of our work is the setting of the Grigorchuk space of marked groups. In this space, finitely presented simple groups are isolated points, and in particular, they cannot emerge as nontrivial limits. By showing that each $G_{\rho}$ group emerges as a nontrivial limit in this setting, we prove our first result:

\begin{theorem} \label{notfp}
For each quasi-periodic labelling $\rho$, the group $G_\rho$ is not finitely presentable.
\end{theorem}

A discrete group $G$ has property $\textup{(T)}$ if every affine isometric action of $G$ on a real Hilbert space admits a fixed point.
Such groups are also called Kazhdan groups.
A group $G$ is said to have Serre's property $\textup{(FA)}$ if every action of $G$ on a simplicial tree has a global fixed point \cite{serretrees}.
It is known that Kazhdan's property $\textup{(T)}$ implies~$\textup{(FA)}$ \cite{watatani}, but the converse is not true. We prove:

\begin{theorem} \label{notkazhdan}
$G_\rho$ does not have Kazhdan's property $\textup{(T)}$, yet it has Serre's property $\textup{(FA)}$.
\end{theorem}

It is a result of Shalom that Kazhdan groups in the space of marked groups form an open set (see \cite{Stalder}). As mentioned above, in this paper we show that the groups $G_{\rho}$ emerge as nontrivial limits. Indeed, the sequence of groups which we demonstrate limit to $G_{\rho}$ are known to not have Kazhdan's property $(T)$, which establishes the first half of the above.

A group $G$ is said to be inner amenable if the action of $G$ by conjugation on $G$ admits an \emph{atomless invariant mean}: a finitely additive probability measure $\mu$ on all subsets of $G$ satisfying $\mu(g^{-1} A g)=\mu(A)$ and which doesn't concentrate on a single element.
In their $1943$ article, Murray and von Neumann introduced property Gamma for von Neumann algebras, which requires the existence of nontrivial asymptotially central sequences. In a $1975$ article, Effros proved that if a group von Neumann algebra factor $LG$ has property Gamma, then $G$ is inner amenable. It was a longstanding question whether there is a counterexample to the converse which has infinite conjugacy classes for each nontrivial element (i.e. the ICC property). A counterexample was found by Vaes in \cite{Vaes}. It is natural in this context to inquire whether the groups $G_{\rho}$ (naturally having the ICC property) are inner amenable and whether $LG_{\rho}$ has property Gamma. We prove the following:

\begin{theorem} \label{notinner}
$G_\rho$ is not inner amenable.
\end{theorem}

For any countable group $G$ there is a sequence of numerical invariants
$\b^{(2)}_0(G), \b^{(2)}_1(G), \ldots \in [0, \i]$ called the $\ell^2$-Betti numbers of $G$.
These are important invariants in geometric group theory. The definition is rather involved
and we will not reproduce it here. We refer the reader to \cite{luck2002l2} for a comprehensive treatment
and to \cite{kammeyer2019introduction} for an introduction with modest prerequisites.

Our interest in the first $\ell^2$-Betti number arises in view of the Osin-Thom conjecture \cite{osin2013normal}.
Recall that a group $G$ is {\em normally generated} by a subset $X \subseteq G$ if $G$ coincides with the normal closure of $X$,
i.e., the only normal subgroup of $G$ containing $X$ is $G$ itself. The {\em normal rank} of $G$,
denoted $\mathrm{nrk}(G)$, is the minimal number of normal generators of $G$. Osin and Thom conjectured that
the inequality $\b^{(2)}_1(G) \leq \mathrm{nrk}(G) - 1$ holds for any torsion free countable group $G$.
If true, that would have important consequences; notably, it would
provide a counterexample to the question (posed by Wiegold) whether every finitely generated perfect group is normally generated by a single element.
Osin and Thom proved the conjecture for groups which are limits (in the space of marked groups) of left-orderable amenable groups.
The theorem below implies that the conjecture holds for $G_\rho$. Note that $G_{\rho}$ cannot emerge as a limit of left orderable amenable groups. Since they are finitely generated and simple, they are non-indicable by a Theorem of Witte-Morris \cite{WM}, whereas limits of finitely generated indicable groups are indicable.

\begin{theorem} \label{l2betti}
The first $\ell^2$-Betti number of $G_\rho$ is equal to 0.
\end{theorem}

A group $G$ is called $N$-uniformly simple if for every non-trivial $f,g \in G$
the element $f$ is the product of at most $N$ elements from $C_{g^{\pm 1}}$, where $C_g$ is the conjugacy class of the element $g$. A group is uniformly simple if it is
$N$-uniformly simple for some natural number $N$ (see \cite{prox} for discussion and further references).
We prove:

\begin{theorem} \label{UniformlySimpleMainTheorem}
$G_\rho$ is uniformly simple.
\end{theorem}

\section{Preliminaries} \label{Prel}

Suppose $E$ is a 1-dimensional manifold and $G$ acts on $E$ by orientation preserving homeomorphisms; for example, $E$ could equal $\bR$ or $[0, 1]$, and $G < \mathrm{Homeo}^+(E)$.
Then we denote
\[
\Supp(g) = \{ x \in E : x \cdot g \neq x \}
\]

The set of all dyadic rationals will be denoted $\bZ[\tfrac{1}{2}]$.

Thompson's groups F and T are thoroughly discussed in \cite{CFP}. We recall the definitions and some facts
we'll need.
\begin{definition}
{\em Thompson's group $F$} is the group of piecewise linear homeomorphisms $f \colon [0, 1] \to [0, 1]$ which satisfy the following conditions
\begin{enumerate}
\item
$f$ is differentiable except at finitely many dyadic rationals.
\item
Wherever $f'$ exists, the value of the derivative is an integer power of 2.
\end{enumerate}
\end{definition}
Identify the circle $\mathbf{S}^1$ as $[0, 1] / \{ 0, 1 \}$, that is, the closed unit interval with endpoints glued together.
If $f \colon \mathbf{S}^1 \to \mathbf{S}^1$ is a homeomorphism, and $x \in \mathbf{S}^1$ is not equal to the endpoint (0 glued together with 1), we can talk about the existence and value
of the derivative $f'(x)$ defined in the obvious way. 
%For simplicity, we'll adopt the convention
%that the derivative never exists at the endpoint.
\begin{definition}
{\em Thompson's group T} is the group of piecewise linear homeomorphisms $f \colon \mathbf{S}^1 \to \mathbf{S}^1$,
such that
\begin{enumerate}
\item
$f$ is differentiable except at finitely many points.
\item
Wherever the derivative exists, its value is an integer power of 2.
\item $X\cdot f=X$ where $X=\mathbf{Z}[\frac{1}{2}]/\mathbf{Z}$.
%For every $x \in C$ where $f'$ does not exist, both $x$ and $x \cdot f$ are dyadic rationals.
\end{enumerate}
\end{definition}
Immediately from the above, we notice that $F$ can be identified with the subgroup of $T$ which fixes the endpoint.
%One can show that $F$ is generated by 2 elements.
The group $F$ admits the finite presentation $\langle f,g\mid [fg^{-1},g^f], [fg^{-1},g^{f^2}]\rangle$.
If $r, s \in \bZ[\tfrac{1}{2}] \cap [0, 1]$ and $r < s$, we denote by $F_{[r, s]}$ the subgroup of $F$
consisting of those elements whose support is included in $[r, s]$. For any such $r$ and $s$, the group
$F_{[r, s]}$ is isomorphic to $F$.
For any group $G$, we denote by $G' = [G, G]$ the commutator subgroup of $G$. One can show that $F'$ is simple, and consists of exactly
those elements $g \in F$ for which $\cl{\Supp(g)} \subseteq (0, 1)$ \cite{CFP}. 

Next, we describe the construction of $G_\rho$. We will closely follow the discussion in \cite{HydeLodha} and \cite{uniformlyperfect}

For any $n \in \bZ$, let $\iota_n \colon [n, n+1) \to (n, n+1]$ be the unique orientation-reversing isometry.
Then, define the map $\iota \colon \bR \to \bR$ via
\[
x \cdot \iota = x \cdot \iota_n, \quad \text{where } x \in [n, n+1), \text{ for any } n \in \bZ
\]
\begin{definition} \label{groupH}
We fix an element $c_0 \in F$ with the following properties:
\begin{enumerate}
\item
The support of $c_0$ equals $\left(0, \tfrac{1}{4} \right)$ and $x \cdot c_0 > x$ for every $x \in \left( 0, \tfrac{1}{4} \right)$.
\item
The restriction $\restr{c_0}{\left(0, \tfrac{1}{16} \right)}$ equals to the map $t \mapsto 2t$.
\end{enumerate}
Let
\[
c_1 = \iota_0 \circ c_0 \circ \iota_0, \quad \nu_1 = c_0 c_1
\]
Note that $\nu_1$ is a \emph{symmetric element} of $F$, which means that it commutes with $\iota_0$. We define a subgroup $H$ of $F$ as
\[
H = \< F', \nu_1 \>
\]
Finally, we fix
\[
\nu_2, \nu_3 \colon [0, 1] \to [0, 1]
\]
as chosen homeomorphisms whose supports are included in $\left( \tfrac{1}{16}, \tfrac{15}{16} \right)$
and that generate the group $F_{\left[ \frac{1}{16}, \frac{15}{16} \right]}$.
\end{definition}
\begin{lemma}[Lemma 2.4 in \cite{HydeLodha}] \label{lemmaH}
The group $H$ generated by $\nu_1, \nu_2, \nu_3$ satisfies:
\begin{enumerate}
\item $H'$ is simple.
\item $H'$ consists of precisely the set of elements of $H$ (or $F$) that are compactly supported
in $(0, 1)$. In particular, $H' = F'$.
\end{enumerate}
\end{lemma}
\begin{definition}
Denote $\frac{1}{2} \bZ = \left\{ \tfrac{1}{2} k : k \in \bZ \right\}$. A {\em labelling} is a map
\[
\rho \colon \tfrac{1}{2}\bZ \to \left\{ a, b, a^{-1}, b^{-1} \right\},
\]
which satisfies
\begin{enumerate}
\item
$\rho(k) \in \left\{ a, a^{-1} \right\}$ for each $k \in \bZ$.
\item
$\rho(k) \in \left\{ b, b^{-1} \right\}$ for each $k \in \tfrac{1}{2} \bZ \setminus \bZ$.
\end{enumerate}
\end{definition}
We regard $\rho$ as a bi-infinite word with respect to the usual ordering of the real numbers. A subset $X \subseteq \tfrac{1}{2}\bZ$ is said
to be a {\em block} if it is of the form
\[
\left\{ k, k + \tfrac{1}{2}, \ldots, k + \tfrac{1}{2}n  \right\}
\]
for some $k \in \tfrac{1}{2} \bZ$, $n \in \bN$. The set of all blocks is denoted by $\bB$. To each block
\[
X = \left\{ k, k + \tfrac{1}{2}, \ldots, k + \tfrac{1}{2}n  \right\}
\]
we assign a formal word
\[
W_\rho(X) = \rho(k) \rho(k + \tfrac{1}{2}) \cdots \rho(k + \tfrac{1}{2}n)
\]
which is a word in the letters $a$, $b$, $a^{-1}$, $b^{-1}$. Such a formal word is called a {\em subword} of the labelling.

Given a word $w_1 \cdots w_n$, the formal inverse
of that word is $w_n^{-1} \cdots w_1^{-1}$, with the natural convention that $\left(a^{-1}\right)^{-1} = a$ and $\left(b^{-1}\right)^{-1} = b$.
The formal inverse of $W_\rho(X)$ is denoted by $W_\rho^{-1}(X)$.

A labelling $\rho$ is said to be {\em quasi-periodic} if the following conditions hold:
\begin{enumerate}
\item
For each block $X \in \bB$, there is an $n \in \bN$ such that whenever $Y \in \bB$ is a block of size at least $n$, then $W_\rho(X)$
is a subword of $W_\rho(Y)$.
\item
For each block $X \in \bB$, there is a block $Y \in \bB$ such that $W_\rho(Y) = W_\rho^{-1}(X)$.
\end{enumerate}
Note that by {\em subword} in the above we mean a string of consecutive letters in the word.

\begin{definition}
Let $H < \mathrm{Homeo}^+([0, 1])$ be the group defined in Definition \ref{groupH}. Recall from Lemma \ref{lemmaH}
that the group $H$ is generated by the three elements $\nu_1, \nu_2, \nu_3$ defined in Definition \ref{groupH}.
In what appears below, by $\cong_T$ we mean that the restrictions are topologically conjugate via the unique
orientation-preserving isometry that maps $[0, 1]$ to the respective interval. We define the homeomorphisms
\[
\zeta_1, \zeta_2, \zeta_3, \chi_1, \chi_2, \chi_3 \colon \bR \to \bR
\]
as follows: for each $i \in \{1, 2, 3\}$ and $n \in \bZ$,
\begin{align*}
\restr{\zeta_i}{[n, n + 1]} \cong_T \nu_i \quad & \text{ if } \rho(n + \tfrac{1}{2}) = b, \\
\restr{\zeta_i}{[n, n + 1]} \cong_T (\iota \circ \nu_i \circ \iota) \quad & \text{ if } \rho(n + \tfrac{1}{2}) = b^{-1}, \\
\restr{\chi_i}{\left[n - \tfrac{1}{2}, n + \tfrac{1}{2}\right]} \cong_T \nu_i \quad & \text{ if } \rho(n) = a, \\
\restr{\chi_i}{\left[n - \tfrac{1}{2}, n + \tfrac{1}{2}\right]} \cong_T (\iota \circ \nu_i \circ \iota) \quad & \text{ if } \rho(n) = a^{-1}
\end{align*}
Note that $\zeta_i$ and $\chi_i$ depend on $\rho$, even though we don't make that explicit in the notation.

The group $G_\rho$ is defined as:
\[
G_\rho = \< \zeta_1, \zeta_2, \zeta_3, \chi_1, \chi_2, \chi_3 \> < \mathrm{Homeo}^+(\bR)
\]
We will denote the tuple $\left( \zeta_1, \zeta_2, \zeta_3, \chi_1, \chi_2, \chi_3 \right)$
by $S_\rho$ and call it the {\em standard generator tuple} for $G_\rho$.
We also define the subgroups
\[
\cK = \< \zeta_1, \zeta_2, \zeta_3 \>, \quad \cL = \< \chi_1, \chi_2, \chi_3 \>
\]
of $G_\rho$, which are both isomorphic to $H$. The isomorphisms will be denoted by $\lambda \colon H \to \cK$ and
$\pi \colon H \to \cL$, and are defined by requiring that for each $f \in H$ and $n \in \bZ$,
\begin{align*}
\restr{\lambda(f)}{[n, n+1]} \cong_T f \quad & \text{ if } \rho(n + \tfrac{1}{2}) = b, \\
\restr{\lambda(f)}{[n, n+1]} \cong_T (\iota \circ f \circ \iota) \quad & \text{ if } \rho(n + \tfrac{1}{2}) = b^{-1}, \\
\restr{\pi(f)}{\left[ n - \tfrac{1}{2}, n + \tfrac{1}{2} \right]} \cong_T f \quad & \text{ if } \rho(n) = a, \\
\restr{\pi(f)}{\left[ n - \tfrac{1}{2}, n + \tfrac{1}{2} \right]} \cong_T (\iota \circ f \circ \iota) \quad & \text{ if } \rho(n) = a^{-1}
\end{align*}
\end{definition}
It follows from the above that
\[
\cK' \cong \cL' \cong H' \cong F'
\]
Note that $\cK$ and $\cL$ also depend on the choice of $\rho$, even though we don't make it explicit in the notation.

The following theorem is proved in \cite{HydeLodha}
\begin{theorem}
Let $\rho$ be a quasi-periodic labelling. Then the group $G_\rho$ is simple.
\end{theorem}

The space of marked groups was first introduced by Grigorchuk in \cite{grigorchuk1985degrees}. The definition we recall below
is similar to the one given in \cite{goldbring2022ultrafilters}.
   
%Fix a countable set $X$. 

\begin{definition}
We will call a pair $(G, S)$ a {\em marked group} if $G$ is a group
and $S = (s_1, \ldots, s_n)$ is a finite ordered tuple such that $\{ s_1, \ldots, s_n \}$ generates $G$.
Let $\tilde{\cG}_n$ denote the set of marked groups $(G, S)$ such that $S$ is an $n$-tuple.
If $(G, S)$ and $(G', S')$ belong to $\tilde{\cG_n}$, then $(G, S)$ and $(G', S')$ are {\em marked isomorphic} if the map $s_i \mapsto s'_i$
for $i = 1, \ldots, n$ extends to an isomorphism between $G$ and $G'$.
We then define the marked isomorphism relation generated by $(G, S)\sim (G', S')$,
and define the quotient space $\cG_n=\tilde{\cG}_n/\sim$.
\end{definition}

\begin{definition}
For any $(G, S), (G', S') \in \tilde{\cG_n}$, let $\nu((G, S), (G', S'))$ denote the maximal $k \in \bN\cup \{\infty\}$ such that,
for any word $w(x_1, \ldots, x_n)$ of length at most $k$, $w(s_1, \ldots, s_n) = e$ in $G$ if and only if
$w(s'_1, \ldots, s'_n) = e$ in $G'$. We then define $d((G, S), (G', S')) = 2^{-\nu((G, S),(G', S'))}$.
One can easily check
that the function 
\[
d \colon \cG_n \x \cG_n \to [0, \i) \colon ([(G, S)], [(G', S')]) \mapsto d((G, S), (G', S'))
\]
is well-defined
and is a metric on $\cG_n$. We call $\cG_n$ equipped with such metric the {\em space of $n$-generated marked groups}.
\end{definition}

For any $x \in \bR$, $n \in \bN$, and a labelling $\s$, let $\cW_\s(x, n)$ denote the word
of length $2n+1$ in the alphabet $\{ a, a^{-1}, b, b^{-1} \}$ defined as follows. Let $y \in \frac{1}{2}\bZ \setminus \bZ$
be such that $x \in \left[y - \frac{1}{2}, y + \frac{1}{2}\right)$. Then
\[
\cW_\s(x, n) = \s (y - \tfrac{1}{2}n) \, \s (y - \tfrac{1}{2}(n - 1)) \cdots \s(y) \cdots \s(y + \tfrac{1}{2}(n - 1)) \, \s(y + \tfrac{1}{2}n)
\]
Next, for any $n \in \bZ$, let $\iota_n$ be the unique orientation-reversing isometry $\iota_n \colon {[n, n + 1)} \to {(n, n + 1]}$. Then define
$\iota \colon \bR \to \bR$ as
\[
x \cdot \iota = x \cdot \iota_n, \quad \text{ where $x \in [n, n + 1)$ for $n \in \bZ$}
\]
Similarly, if $I$ is a compact interval with endpoints in $\mathbf{Z}[\frac{1}{2}]$ and $n \in \bN$, we define:
\[
\cW_\s(I, n) = \s (\inf(I) - \tfrac{1}{2}n) \, \s (\inf(I) - \tfrac{1}{2}(n - 1)) \cdots \s(\sup(I) + \tfrac{1}{2}(n - 1)) \, \s(\sup(I) + \tfrac{1}{2}n)
\]

We recall the following \emph{characterization of elements} of $G_{\rho}$.
This provides an alternative, ``global" description of the groups as comprising of elements satisfying dynamical and combinatorial hypotheses,
and is reminiscent of similar descriptions for various generalisations of Thompson's groups.

\begin{definition}\label{Krho}
Let $K_{\rho}$ be the set
of homeomorphisms $f\in \textup{Homeo}^+(\mathbf{R})$ satisfying the following:
\begin{enumerate}
\item $f$ is a countably singular piecewise linear homeomorphism of $\mathbf{R}$ with a discrete set of singularities, all of which lie in $\mathbf{Z}[\frac{1}{2}]$.
\item $f'(x)$, wherever it exists, is an integer power of $2$.
\item There is a $k_f\in \mathbf{N}$ such that the following holds.
\begin{enumerate}
\item[3.a] Whenever $x,y\in \mathbf{R}$ satisfy that $$x-y\in \mathbf{Z} \text{ and }  \mathcal{W}(x,k_f)=\mathcal{W}(y,k_f),$$ it holds that $$x-x\cdot f = y-y\cdot f.$$
\item[3.b] Whenever $x,y\in \mathbf{R}$ satisfy that $$x-y\in \mathbf{Z}\text{ and } \mathcal{W}(x,k_f)=\mathcal{W}^{-1}(y,k_f),$$ it holds that $$x-x\cdot f=y'\cdot f-y' \qquad \text{ where }y'=y\cdot \iota_{[n,n+1]}.$$
where $n\in \mathbf{N}$ is such that $y\in [n,n+1)$.
\end{enumerate}
\end{enumerate}
\end{definition}

The following is Theorem $1.8$ in \cite{uniformlyperfect}.
\begin{theorem}\label{characterisation}
Let $\rho$ be a quasi-periodic labelling.
The groups $K_{\rho}$ and $G_{\rho}$ coincide.
\end{theorem}

%Note that it is possible that $\nu((G, S), (G', S')) = \i$, from which $d((G, S), (G', S')) = 0$. This happens if and only if
%the marked groups are isomorphic in the sense of the following definition:

%We then define $\cG_n$ to be the set of isomorphism classes of marked groups with an $n$-tuple of generators. \

The following is Proposition 4.16 in \cite{ffflodha}, which shall also be an ingredient in our proofs.

\begin{proposition}\label{Fprimefff}
Let $\rho$ be a quasi-periodic labelling.
Let $f_1, \ldots, f_n\in G_{\rho}$ be elements such that there exists an open interval $I$ which is pointwise fixed by the element $f_i$ for each $1\leq i\leq n$.
Then there is a subgroup $A<G_{\rho}$ isomorphic to a finite direct sum of copies of $F'$, which contains $f_1, \ldots, f_n$.
\end{proposition}

The following Lemma provides a construction of elements in $G_{\rho}$.

\begin{prop}\label{SpecialElements2}
Let $\rho$ be a quasi-periodic labelling and $I$ a compact interval with integer endpoints. 
Consider the word $\mathcal{W}_{\rho}(I,n)$ for some $n\in \mathbf{N}\setminus \{0\}$.
Fix the natural injective homomorphism $\phi_I:F'\to \textup{Homeo}^+(I)$ obtained by the topological conjugacy mapping $[0,1]$ to $I$ linearly. 
For each $f\in F'$, there is an element $g\in G_{\rho}$ that satisfies the following for each $x\in \mathbf{R}$.
\begin{enumerate}
\item If there is a compact interval $J$ containing $x$, with integer endpoints such that $|J|=|I|$
and $\mathcal{W}_{\rho}(J,n)=\mathcal{W}_{\rho}(I,n)$, 
then $$x\cdot g-x=x'\cdot \phi_I(f)-x'\qquad x'=x\cdot T_{J,I}$$  
\item If there is a compact interval $J$ containing $x$, with integer endpoints such that $|J|=|I|$
and $\mathcal{W}_{\rho}(J,n)^{-1}=\mathcal{W}_{\rho}(I,n)$,
then $$x\cdot g-x=x'-x'\cdot \phi_I(f)\qquad x'=x\cdot T_{J,I}^{or}$$  
\item If neither of the above is satisfied, then $x\cdot g=x$. 
\end{enumerate}
\end{prop}

\begin{proof}
The element $g$ lies in $G_{\rho}$ since it lies in $K_{\rho}$,
satisfying the hypothesis of Definition \ref{Krho} with $k_g=2|\mathcal{W}(I,n)|$.
\end{proof}

\section{Proof that $G_\rho$ is not finitely presented}

\begin{lemma} \label{glocal}
Let $\s$ and $\t$ be any two labellings and $G_\s$, $G_\t < \mathrm{Homeo}^+(\bR)$ the corresponding
groups of homeomorphisms of the reals. Let $S_\s$ and $S_\t$ denote the standard generator tuples for $G_\s$ and $G_\t$, respectively.
Let $w(x_1, \ldots, x_6)$ be an arbitrary formal word on $x_1^{\pm 1}, \ldots, x_6^{\pm 1}$. If $S = (s_1, \ldots, s_6)$
we will write $w(S) = w(s_1, \ldots, s_6)$. Suppose $n \in \bN$ is such that $\lv w \rv \leq n$. Finally, let $x, y \in \bR$ be any two reals.

If $m = y - x \in \bZ$ and $\s(\a) = \t(\a + m)$ for all $\a \in \frac{1}{2} \bZ \cap (x - n, x + n)$, then
$x \cdot w(S_\s) - x = y \cdot w(S_\t) - y$.
\end{lemma}
\begin{proof}
Let us denote $S_\s = (s^\s_1, \ldots, s^\s_6)$ and $S_\t = (s^\t_1, \ldots, s^\t_6)$. To be clear, these are the same generator tuples which
in case of a fixed quasi-periodic labelling are more traditionally denoted $(\z_1, \z_2, \z_3, \chi_1, \chi_2, \chi_3)$.
Directly from the definition, for any $x \in \bR$ and $j = 1, \ldots, 6$, the distance between $x$ and $x \cdot s^\s_j$ can be at most 1, and same for $s^\t_j$.
For any $x, y \in \bR$, we can define blocks $X = \frac{1}{2} \bZ \cap \left[x - \frac{1}{2}, x + \frac{1}{2}\right]$
and $Y = \frac{1}{2} \bZ \cap \left[y - \frac{1}{2}, y + \frac{1}{2} \right]$.
If $x - y \in \bZ$ and $\s(X) = \t(Y)$, then $x \cdot s^\s_j - x = y \cdot s^\t_j - y$, again from the definition.
The desired conclusion is then a matter of inducting on $n$.
\end{proof}

\begin{lemma} \label{subwords}
Let $\rho$ be a quasi-periodic labelling, and let $k \in \bN$ be arbitrary. We can find a periodic labelling $\s$ such that
for any block $X \subset \frac{1}{2}\bZ$ of length $k$, there exist blocks $Y, Z \subset \frac{1}{2}\bZ$ satisfying
$\rho(X) = \s(Y)$ and $\s(X) = \rho(Z)$.
\end{lemma}
\begin{proof}
There are only finitely many subwords of $\rho$ of length $k$, and we can find a block
$X = \left\{ n, n + \frac{1}{2}, \ldots, n + \frac{1}{2} m \right\}$ such that the word
$\rho(X) = \rho(n) \rho({n + \frac{1}{2}}) \cdots \rho({n + \frac{1}{2}m})$ will contain them all as subwords. That is, if $Y \subset \frac{1}{2} \bZ$ is a block of size $k$, then there exists a block $Y' \subseteq X$
such that $\rho(Y) = \rho(Y')$. Let $Z = \left\{ n, n + \frac{1}{2}, \ldots, n + \frac{1}{2}(k - 1) \right\}$
be the block containing the first $k$ elements of $X$. By the quasi-periodic nature of $\rho$,
we can find another block $Z' = \left\{ \ell, \ell + \frac{1}{2}, \ldots, \ell + \frac{1}{2}(k - 1) \right\}$
with $\ell > n + \frac{1}{2} m$, such that $\rho(Z) = \rho(Z')$. Let
$X' = \left\{ n, n + \frac{1}{2}, \ldots, \ell - \frac{1}{2} \right\}$.
Then define $\s$ to be the unique periodic labelling with period $\lv X' \rv$ which agrees with $\rho$ on $X'$.
(I.e., $\s$ is just a string of copies of $\rho(X')$ repeating one after another in both directions.)

Any subword $w$ of $\rho$ of length $k$ appears as a subword in $\rho(X)$, which is a subword of $\rho(X')$, and so $w$ is also a subword of $\s$.
Conversely, let $v$ be a subword of $\s$ of length $k$. If $v$ is a subword of $X'$ then it is a subword of $\rho$. If on the other hand
$v$ spans the boundary between two consecutive copies of $X'$, then it is also a subword of $\rho$ due to the fact that the first
$k$ letters in $\rho$ immediately following $\rho(X')$ are the same as the first $k$ letters in $\rho(X')$.
\end{proof}

\begin{lemma} \label{limitofperiodic}
Consider $(G_\rho, S_\rho)$ as a point in $\cG_6$, the Grigorchuk's space of marked groups on 6 generators.
There exists a sequence of periodic labellings $\{\rho_n\}$ such that the corresponding sequence of marked groups
$\{ (G_{\rho_n}, S_{\rho_n}) \} \subseteq \cG_6$ satisfies
$(G_{\rho_n}, S_{\rho_n}) \xrightarrow{n \to \i} (G_\rho, S_\rho)$
\end{lemma}
\begin{proof}
Below, we will understand $w(x_1, \ldots, x_6)$ to mean a formal word on $x_1^{\pm 1}, \ldots, x_6^{\pm 1}$.
We will also use $w(S)$ to mean $w(s_1, \ldots, s_6)$ where $S = (s_1, \ldots, s_6)$, and similarly for $S_\rho$.

The following family of sets $\{ B_n \}$, $n \in \bN$, is a neighborhood base at $(G_\rho, S_\rho)$ for the topology on $\cG_6$:
\begin{align*}
B_n = \big\{ (G, S) \in \cG_6 : & \text{ for all words } w(c_1, \ldots, c_6) \text{ of length } \lv w \rv \leq n, \\
& \ w(S) = e \text{ in $G$ if and only if } w(S_\rho) = e \text{ in $G_\rho$} \big\}
\end{align*}
Notice $B_n \subseteq B_m$ whenever $n \geq m$. So it's sufficient if
given $n \in \bN$, we find a periodic labelling $\rho_n$ such that $(G_{\rho_n}, S_{\rho_n}) \in B_n$.
To that end, let $\rho_n$ be the periodic labelling, whose existence is guaranteed by Lemma \ref{subwords} with $k = 4n$, such that
every subword of $\rho$ of length at most $4n$ is also a subword of $\rho_n$ and every subword of $\rho_n$
of length at most $4n$ is also a subword of $\rho$.

Take any $w(s_1, \ldots, s_6)$ of length at most $n$ and suppose $w(S_{\rho}) = e$.
We want to show that $w(S_{\rho_n}) = e$. Assuming otherwise, we take an arbitrary $y \in \bR$ such that $y \cdot w(S_{\rho_n}) - y\neq 0$. Let $Y = \frac{1}{2}\bZ \cap {(y - n, y + n)}$.
Observe that $Y$ is a block of length less than $4n$. By construction, there exists another block $X=\frac{1}{2}\bZ \cap {(x - n, x + n)}$ such that
$\rho_n(Y) = \rho(X)$. Let $m = \min(Y) - \min(X)$ and $y = x + m$. Then $y - x = m \in \bZ$ and $Y = \frac{1}{2}\bZ \cap (y - n, y + n)$.
Since $\rho_n(Y) = \rho(X)$, we have $\rho_n(\a + m) = \rho(\a)$ for every $\a \in \frac{1}{2}\bZ \cap (x - n, x + n)$, so by Lemma
\ref{glocal} it follows that $y \cdot w(S_{\rho_n}) - y = x \cdot w(S_\rho) - x$. But $x \cdot w(S_\rho) - x = 0$, since $w(S_{\rho}) = e$.
Hence also $y \cdot w(S_{\rho_n}) - y = 0$, contradicting our assumption.
By an analogous argument, $w(S_{\rho}) = e$ whenever $w(S_{\rho_n}) = e$, so $(G_{\rho_n}, S_{\rho_n}) \in B_n$ as needed.

\end{proof}

We call a group $G$ {\em finitely discriminable} if there exists a finite subset $F \subseteq G \setminus \{ e \}$ such that every non-trivial
normal subgroup of $G$ contains an element of $F$. In particular, every simple group is finitely discriminable.
Let $G$ be a finitely generated group. Cornulier et al prove in \cite{de2007isolated} that if for some generating set $S$, the pair $(G, S) \in \cG_{\lv S \rv}$
is an isolated point in the space of marked groups, then same is true for every generating set. In such case we call $G$ an {\em isolated group}.
We will need the following characterization of isolated groups (see Proposition 2 in \cite{de2007isolated})
\begin{proposition}[Cornulier, Guyot and Pitsch] \label{isolatedgroups}
A group $G$ is isolated if and only if the following properties are satisfied
\begin{enumerate}
\item [(i)]
$G$ is finitely presented;
\item [(ii)]
$G$ is finitely discriminable.
\end{enumerate}
\end{proposition}

We will need the following Lemma.

\begin{lemma}\label{nonisom}
If $\s$ is a periodic labelling then $G_\s$, then $G_\s \neq G_\rho$.
\end{lemma}

\begin{proof}
If $\s$ is a periodic labelling then all elements of $G_\s$ commute with some suitable positive integer translation. Then $G_\s$ satisfies that the map $\phi:G_{\s}\to \mathbf{R}$ given by $f\mapsto \lim_{n\to \infty} \frac{0\cdot f^n}{n}$ is a nontrivial homogeneous quasimorphism (this sort of a quasimorphism is sometimes called the translation number quasimorphism, see example $2.3$ in \cite{ffflodha} for details, noting that translation by $1$ can be replaced by translation by any positive integer and thereby adjusting the defect in the example). However, for a quasi-periodic labelling $G_\rho$ such nontrivial homogeneous quasimorphisms do not exist since they have vanishing bounded cohomology in degree $2$ with trivial real coefficients (see Theorem $1.5$ in \cite{ffflodha}). 
\end{proof}

We remark that Lemma \ref{lifting} from the subsequent section also serves as an alternative proof of Lemma \ref{nonisom}.

\begin{proof}[Proof of Theorem \ref{notfp}]
If $\s$ is a periodic labelling then by Lemma (\ref{nonisom}), $G_\s \not\cong G_\rho$.
Lemma (\ref{limitofperiodic}) then implies that $(G_\rho, S_\rho)$ is a limit of a sequence in $\cG_6$ whose elements are not isomorphic to $(G_\rho, S_\rho)$,
which means $G_\rho$ is not an isolated group. Since $G_\rho$ is simple and hence finitely discriminable, Proposition (\ref{isolatedgroups}) tells us that $G_\rho$ cannot be
finitely presented.
\end{proof}

\section{Periodic labellings.}
In this section we study the structure of the groups $G_{\rho}$ where $\rho$ is a periodic labelling.  
We will need the following slight generalization of Thompson's group~$T$.
\begin{definition}
Let $n \in \bZ_+$. Let $C(n) = [0, n] / \{0, n\}$, that is, $C(n)$ is constructed from the closed interval $[0, n]$ by gluing together the endpoints.
Define $T(n)$ to be the group of piecewise linear homeomorphisms $f \colon C(n) \to C(n)$ which satisfy
\begin{enumerate}
\item
$f$ is differentiable except at finitely many points.
\item
For any $x \in C(n)$ where $f'$ does exist, $f'(x)$ is an integer power of 2.
\item
For any $x \in C(n)$ where $f'$ does not exist, both $x$ and $f(x)$ are dyadic rationals.
\item \label{zerodyadic}
$0 \cdot f \in \bZ[\frac{1}{2}]$.
\end{enumerate}
\end{definition}
%(Note that since it's not completely obvious what $f'(x)$ should mean if either $x$ or $f(x)$ equal $0 = n$, we could just declare
%that $f$ is never differentiable in such case, obviating the need for condition \ref{zerodyadic}).

In particular, $T(1)$ is the same as Thompson's group $T$. More generally, if $n$ is a power of 2 then $T(n)$ is
isomorphic to $T$.

The group $T(n)$ is in most ways similar to $T$. The lemma below is analogous to the fact that $T$ is generated by $F$ and dyadic translations.
\begin{lemma} \label{tgenerated}
For $n \in \bZ_+$, let $F(n)$ be the subgroup of $T(n)$ consisting of those elements of $T(n)$ which fix $0 = n \in C(n)$.
Let $S(n)$ be the subgroup of $T(n)$ consisting of rotations by dyadic rationals between 0 and $n$, that is, $s \in S(n)$
whenever $s \colon x \mapsto \phi_n(x + \a)$ for some $\a \in \bZ[\frac{1}{2}] \cap [0, n]$, where $\phi_n:\bZ\to C(n)\cong\bR/n\bZ$ is the canonical map. Then every $t \in T(n)$ can be written
as $t = sf$ with $s \in S(n)$ and $f \in F(n)$.
\end{lemma}
\begin{proof}
For $t \in T(n)$, it is enough to take $s$ to be the translation by $0 \cdot t$, and $f = s^{-1}t$.
\end{proof}

We will want to show that for a periodic labelling $\s$, the group $G_\s$ is a lift of $T(n)$ for some $n$. To that end, we will need a modified version of a
theorem proven by Hyde et al. (See Theorem 1.8 in \cite{uniformlyperfect}. Note that the statement of the theorem in \cite{uniformlyperfect} requires the labelling to be quasi-periodic,
as opposed to periodic which we have below.
Also note that condition (4) is absent in the original statement,
since for quasi-periodic labellings it follows from the other conditions.)

We recall the following definition from the Preliminaries.
\begin{definition}[Compare definition 1.6 in \cite{uniformlyperfect}] \label{Kset}
For any labelling $\s$, let $K_\s$ be the set of homeomorphisms $f \in \mathrm{Homeo}^+(\bR)$ satisfying
the following conditions.
\begin{enumerate}
\item [(1)]
$f$ is a piecewise linear homeomorphism of $\bR$ with a discrete set of breakpoints, all of which lie in $\bZ[\frac{1}{2}]$.
\item [(2)]
$f'(x)$, wherever it exists, is an integer power of 2.
\item [(3)]
There is a $k_f \in \bN$ such that:
\begin{enumerate}
\item [(3.a)]
whenever $x, y \in \bR$ satisfy
\[
x - y \in \bZ, \quad \cW_\s(x, k_f) = \cW_\s(y, k_f)
\]
we have
\[
x - x \cdot f = y - y \cdot f
\]
\item [(3.b)]
whenever $x, y \in \bR$ satisfy
\[
x - y \in \bZ, \quad \cW_\s(x, k_f) = \cW_\s^{-1}(y, k_f)
\]
we have
\[
x - x \cdot f = y' \cdot f - y', \quad \text{ where $y' = y \cdot \iota$}
\]
\end{enumerate}
\item [(4)]
$0 \cdot f \in \bZ[\frac{1}{2}]$
\end{enumerate}
\end{definition}
Note that in light of Lemma \ref{periodicsubwords} below, condition (3) could be greatly simplified for periodic labellings.
We choose to keep the original condition to emphasise the analogy with the quasi-periodic case.
\begin{theorem}[Compare Theorem 1.8 in \cite{uniformlyperfect}] \label{gbig}
If the labelling $\s$ is periodic, then $K_\s = G_\s$.
\end{theorem}

The heavy lifting in the proof of Theorem \ref{gbig} is done by Proposition 3.4 in \cite{uniformlyperfect}. To state it, we need
the following definitions.

Below, we will use notation $\cW(J, n)$, where $n \in \bN$ and $J = [m_1, m_2] \subseteq \bR$ is a compact interval with integer endpoints,
to describe a word defined as
\begin{multline*}
\cW(J, n) = \rho \! \left(m_1 - \tfrac{1}{2}(n - 1) \right) \, \rho \! \left( m_1 - \tfrac{1}{2}(n - 2) \right) \cdots \\
\cdots \rho \! \left(m_2 + \tfrac{1}{2}(n - 2)\right) \, \rho \! \left(m_2 + \tfrac{1}{2}(n - 1) \right)
\end{multline*}

\begin{definition}[Definition 3.1 in \cite{uniformlyperfect}]
A homeomorphism $f \in \mathrm{Homeo}^+(\bR)$ is said to be {\em stable} if there exists an
$n \in \bN$ such that the following condition holds.
For any compact interval $I$ of length at least $n$, there is an integer $m \in I$ such that $f$
fixes a neighborhood of $m$ pointwise. Given a stable homeomorphism $f \in \mathrm{Homeo}^+(\bR)$
and an interval $[m_1, m_2]$, the restriction $\restr{f}{[m_1, m_2]}$ is said to be an {\em atom of $f$}, if:
\begin{enumerate}
\item
$m_1, m_2 \in \bZ$;
\item
there is an $\e > 0$ such that, for each $x \in (m_1 - \e, m_1 + \e) \cup (m_2 - \e, m_2 + \e)$,
we have $x \cdot f = x$;
\item
for any $m \in (m_1, m_2) \cap \bZ$ and any $\e > 0$, there is a point $x \in (m - \e, m + \e)$
such that $x \cdot f \neq x$.
\end{enumerate}
In other words, an atom is the restriction of $f$ to the closure of a maximal open interval $J$
with the property that for each $m \in J \cap \bZ$, $f$ moves a point in any neighborhood of $m$.

Given a stable homeomorphism $f$, there is a unique way to express $\bR$ as a union of integer endpoint intervals
$\{ I_\a \}_{\a \in P}$ such that $\restr{f}{I_\a}$ is an atom for each $\a \in P$ and different
intervals intersect in at most one endpoint. We will also refer to the intervals $I_\a$ as the atoms of $f$.

Two atoms $\restr{f}{[m_1, m_2]}$ and $\restr{f}{[m_3, m_4]}$ are said to be {\em conjugate} if there is an integer translation
$h(t) = t + z$ for $z=m_3-m_1 \in \bZ$ such that
\[
\restr{f}{[m_1, m_2]} = \restr{h^{-1} \circ f \circ h}{[m_3, m_4]}
\]
and {\em flip-conjugate} if there is an integer translation $h(t) = t + z$ for $z=m_3-m_1 \in \bZ$ such that
\[
\restr{f}{[m_1, m_2]} = \restr{h^{-1} \circ \left( \iota_{[m_1, m_2]} \circ f \circ \iota_{[m_1, m_2]} \right) \circ h}{[m_3, m_4]}
\]
where $\iota_{[m_1, m_2]} \colon [m_1, m_2] \to [m_1, m_2]$ is the unique orientation-reversing isometry.

A \emph{decorated atom for $f$} is simply a pair $(I_{\alpha},n)$ where $I_{\alpha}$ is an atom for $f$ and $n\in \mathbf{N}$.
For a fixed $n \in \bN$, we consider the set of {\em decorated atoms}:
\[
\cT_n(f) = \left\{ (I_\a, n) : \a \in P \right\}
\]
We say that a pair of decorated atoms $(I_\a, n)$ and $(I_\b, n)$ are equivalent if either of the following statements holds:
\begin{enumerate}
\item
$I_\a$, $I_\b$ are conjugate and $\cW(I_\a, n) = \cW(I_\b, n)$;
\item
$I_\a$, $I_\b$ are flip-conjugate and $\cW(I_\a, n) = \cW^{-1}(I_\b, n)$.
\end{enumerate}

The element $f$ is said to be {\em uniformly stable} if it is stable and there are finitely many equivalence classes
of decorated atoms for each $n \in \bN$.
\end{definition}

\begin{definition}[Definition 3.3 in \cite{uniformlyperfect}]
Let $f \in \mathrm{Homeo}^+(\bR)$ be uniformly stable. Let $\z$ be an equivalence class of elements in $\cT_n(f)$.
We define the homeomorphism $f_{\z}$ as
\begin{align*}
& \restr{f_\z}{I_\a} = \restr{f}{I_\a} \quad \text{ if $(I_\a, n) \in \z$}, \\
& \restr{f_\z}{I_\a} = \restr{\id}{I_\a} \quad \text{ if $(I_\a, n) \not\in \z$}
\end{align*}
If $\z_1, \ldots, \z_m$ are the equivalence classes of elements in $\cT_n(f)$, then the list of homeomorphisms
$f_{\z_1}, \ldots, f_{\z_m}$ is called the {\em cellular decomposition of $f$}.
\end{definition}

\begin{proposition} \label{prop34general}
Let $\s$ be any labelling. Given a uniformly stable element $f \in K_\s$,
there is an $n \in \bN$ such that $f_\z \in G_\s$ for each $\z \in \cT_n(f)$. In particular,
it follows that $f \in G_\s$.
\end{proposition}
\begin{proof}
We refer the reader to the proof of Proposition 3.4 in \cite{uniformlyperfect}.
While the proposition is stated there for quasi-periodic labellings only,
the assumption that the labelling be quasi-periodic is not actually used and the same proof can be used word for word to prove Proposition \ref{prop34general}.
\end{proof}

\begin{lemma} \label{periodicuniformlystable}
If $\s$ is a periodic labelling, then every $f \in K_\s$ which is stable is also uniformly stable.
\end{lemma}
\begin{proof}
Let $\s$ be a periodic labelling with period $2k$. Suppose that $f \in K_\s$ is stable.
Let $J = [m_1, m_2]$ be any atom of $f$. Define $\ell \in \bZ$ to be the largest multiple of $k$ which is no greater than $m_1$.
Clearly $\ell - m_1 \in \{0, \ldots, k - 1\}$. Let $K = [m_1 - \ell, m_2 - \ell]$.
Observe that $\cW(J, n) = \cW(K, n)$ for any $n \in \bN$.
The requirement (3.a) in the definition of $K_\s$ then guarantees 
that $f\restriction J=f\restriction K$. %$(J, n)$ is conjugate to $(K, n)$ for any $n \in \bN$. 
Since there are only finitely many
atoms whose left endpoint lies in $\{0, \ldots, k - 1\}$, we conclude that $f$ is uniformly stable.
\end{proof}

For the subsequent proofs, recall the subgroups
\[
\cK = \< \zeta_1, \zeta_2, \zeta_3 \>, \quad \cL = \< \chi_1, \chi_2, \chi_3 \>
\]
of $G_\rho$, which are both isomorphic to $H$. The isomorphisms are denoted by $\lambda \colon H \to \cK$ and
$\pi \colon H \to \cL$.
Also, recall that
$\cK' \cong \cL' \cong H' \cong F'$. We will make use of the subgroups $\pi(H'),\lambda(H')$ which will be denoted by $\pi(F'),\lambda(F')$.
%Note that $\cK$ and $\cL$ also depend on the choice of $\rho$, even though we don't make it explicit in the notation.

\begin{lemma} \label{dyadictozero}
Let $\s$ be any labelling and let $r$ be any dyadic rational. Then there exists some $g \in G_\s$ such that $r \cdot g = 0$. In particular, the action of $G_{\s}$ on $\mathbf{Z}[\frac{1}{2}]$ is transitive.
\end{lemma}
\begin{proof}
Since the action of $F'$ on $(0,1)$ is minimal, $\lambda(F')$ acts minimally on $(n,n+1)$ for each $n\in \mathbf{Z}$. The same holds for $\pi(F')$ in $G_{\s}$ which acts minimally on $(n,n+1)$ for each $n\in \mathbf{Z}[\frac{1}{2}]\setminus \mathbf{Z}$.
It follows from a straightforward inductive argument that the action of $G_{\s}$ is minimal on $\mathbf{R}$.
So we can find $f\in G_{\s}$ such that $r\cdot f \in (0,\frac{1}{2})$. Since the dyadics are invariant under the action of elements in $G_{\s}$ and the action of $\pi(F')$ in $G_{\s}$ on $\mathbf{Z}[\frac{1}{2}]\cap (-\frac{1}{2},\frac{1}{2})$ is transitive,
we can find $h\in \pi(F')$ such that $r\cdot fh=0$. It follows that $0,r$ lie in the same $G_{\s}$-orbit.
%By definition, there is an element $f\in G_{\s}$ such that $0\cdot f\in (0,1)$, and using the fact that $F'$ acts transitively on the dyadic rationals in $(0,1)$ (in the standard action) we can find an element 
%$g\in G_{\s}$ such that $0\cdot fg=\frac{1}{2}$.
%So it suffices 
%Suppose that $r \geq 0$. (The case of $r \leq 0$ is addressed through a very similar argument).
%Let $k$ be the least element of $\bN$ such that $r \in \left[\frac{1}{2}k, \frac{1}{2}k + \frac{1}{2}\right)$.
%The proof proceeds by induction on $k$.
%If $k = 0$, let $r' = r + \frac{1}{2}$. Then $r' \in (0, 1)$. We can find $f \in F'$ such that $r' \cdot f = \frac{1}{2}$.
%Then $r \cdot \pi(f) = 0$.
%Now consider $k > 0$. If $k$ is even, let $r' = r - \frac{1}{2}(k - 1)$. Then $r' \in \left[\frac{1}{2}, 1\right)$ and we can find
%$f \in F'$ for which $r' \cdot f \in \left(0, \frac{1}{2}\right)$. Then $r \cdot \pi(f) \in \left[\frac{1}%{2}(k - 1), \frac{1}{2}(k - 1) + \frac{1}{2}\right)$.
%By the inductive hypothesis there is some $h \in G_\s$ satisfying $r \cdot \pi(f) \cdot h = 0$.
%Hence it's enough to take $g = \pi(f) h$.
%If $k$ is odd, let $r' = r - \frac{1}{2}k$. Then $r' \in \left[\frac{1}{2}, 1\right)$, as before we find $f \in F'$ for which $r' \cdot f \in \left(0, \frac{1}{2}\right)$.
%Then $r \cdot \lambda(f) \in \left[\frac{1}{2}(k - 1), \frac{1}{2}(k - 1) + \frac{1}{2}\right)$, and
%we conclude by calling the inductive hypothesis in the same manner as before.
\end{proof}

\begin{proof}[Proof of Theorem \ref{gbig}]
To show that $G_\s \subseteq K_\s$, we need to establish that every element of $G_\s$ satisfies conditions (1) \ldots (4) from Definition \ref{Kset}.
We can consult the definition of $G_\s$ to observe that all its generators are piecewise linear homeomorphisms satisfying (1) and (2),
and arbitrary products of such elements will again be piecewise linear homeomorphisms satisfying (1) and (2).
Condition (4) follows from the fact that the action of any generator of $G_\s$ on any real number $x$ is the restriction of
a mapping $x \mapsto a x + b$ with $a, b$ some dyadic rationals.

To verify condition (3), consider arbitrary $f \in G_\s$ and $x, y \in \bR$.
If we denote the standard list of generators of $G_\s$ by $S_\s = (s^\s_1, \ldots, s^\s_6)$, then
$f = w(S_\s)$, where $w(S_\s)$ stands for some formal word on elements of $S_\s$ and their inverses, as in the statement of Lemma \ref{glocal}.
Let $n$ be the length of the word $w$.
We can then deduce (3.a) with $k_f = 2n$ from Lemma \ref{glocal}, using $\s$ for both labellings that appear in the lemma.

For (3.b), define the labelling $\t$ by $\t \colon \frac{1}{2} \bZ \to \{ a, b, a^{-1}, b^{-1}\} \colon \a \mapsto \s(-\a)^{-1}$.
Let $S_\t = (s_1^\t, \ldots, s_6^\t)$ be the standard generators of $G_\t$.
Directly from the definition, for any $x \in \bR$ we have
$x \cdot s_j^\s = - \big( (-x) \cdot s_j^\t \big)$, with j = 1, \ldots, 6, and
consequently $x \cdot w(S_\s) = - \big( (-x) \cdot w(S_\t) \big)$.
Let $k_f = 2n$ as before. Recall that $y' = y \cdot \iota$, then observe that $\cW_\t(-y', k_f) = \cW_\s^{-1}(y, k_f)$.
So from $\cW_\s(x, k_f) = \cW_\s^{-1}(y, k_f)$ we can infer $\cW_\t(-y', k_f) = \cW_\s(x, k_f)$, and then use Lemma \ref{glocal}
again to get $x \cdot w(S_\s) - x = (-y') \cdot w(S_\t) - (-y')$. Since $f = w(S_\s)$ and
$y' \cdot w(S_\s) = - \big( (-y') \cdot w(S_\t) \big)$, we can conclude that
$x \cdot f - x = y' - y' \cdot f$ as needed.

Next, we'll show that $K_\s \subseteq G_\s$.
Let $f \in K_\s$ be arbitrary. Let $r = 0 \cdot f$. We know $r$ is a dyadic rational, so by Lemma \ref{dyadictozero} we can find $g \in G_\s$
for which $r \cdot g = 0$. We can find an element $h \in F'$ such that $\pi(h)$ coincides with $fg$ on some neighborhood of 0.
Then $fg \big( \pi(h) \big)^{-1}$ fixes some neighborhood of 0 pointwise, so it is stable, and hence uniformly stable by
Lemma \ref{periodicuniformlystable}. By Proposition \ref{prop34general} we conclude that $fg \big( \pi(h) \big)^{-1} \in G_\s$.
Since $g \in G_\s$ and $\pi(h) \in G_\s$, it follows that $f \in G_\s$ as desired.
\end{proof}

\begin{lemma} \label{periodicsubwords}
Suppose $\s$ is a periodic labelling with period $n$. Let $k \in \bZ$ be arbitrary.
Let $X = \{ \frac{1}{2} \ell, \frac{1}{2}(\ell + 1), \ldots, \frac{1}{2}(\ell + m) \}$
be a block of size at least $2n$, and let
$Y = \{ \frac{1}{2} \ell + k, \frac{1}{2}(\ell + 1) + k, \ldots, \frac{1}{2}(\ell + m) + k \}$
be another block, formed by adding $k$ to every element of $X$.
Then
\begin{enumerate}
\item
If $n \nmid k$ then $W_\s(X) \neq W_\s(Y)$.
\item
$W_\s(X) \neq W_\s^{-1}(Y)$.
\end{enumerate}
\end{lemma}
\begin{proof}
Since $X$ has size at least $2n$, for any $j \in \frac{1}{2} \bZ$ we can find an element of $j' \in X$
such that $n \mid (j - j')$.
Then $\s(j) = \s(j')$ and $\s(j' + k) = \s(j + k)$.
So whenever $W_\s(X) = W_\s(Y)$, we actually have $\s(j) = \s(j + k)$ for every $j \in \frac{1}{2} \bZ$.
If $n \nmid k$ then there exists some $0 < k' < n$ such that $n \mid (k - k')$. But then
$\s(j + k) = \s(j + k')$, so $W_\s(X) = W_\s(Y)$ would lead to $\s(j) = \s(j + k')$ for all $j$,
which is impossible because $k'$ is less than the period of $\s$. That proves part 1.

For part 2, suppose to the contrary, that $W_\s(X) = W_\s^{-1}(Y)$. Since $k \in \bZ$, either both $\frac{1}{2}\ell$ and $\frac{1}{2}\ell + k$
are in $\bZ$, or both are in $\frac{1}{2}\bZ \setminus \bZ$, so either both $W_\s(X)$ and $W_\s(Y)$ begin with $a^{\pm 1}$, or they both begin
with $b^{\pm 1}$. Then they must also both end with $a^{\pm 1}$ in the former case or $b^{\pm 1}$ in the latter,
since otherwise they couldn't be inverses of one another. It follows that the words $W_\s(X)$ and $W_\s(Y)$ are of odd size.

Next, we can find $k' \in \bZ$ such that $n \nmid (k - k')$ and $k' \in \{ 0, 1, \ldots, n - 1 \}$.
Since $\s$ is periodic, it contains a subword which is a copy of $W_\s(Y)$ starting at $\frac{1}{2}\ell + k'$.
Recall that $m$, the size of both $X$ and $Y$, is at least $2n$.
So the word $\cW_* = \s(\frac{1}{2}\ell + k') \s(\frac{1}{2}(\ell + 1) + k') \cdots \s(\frac{1}{2}(\ell + m))$
is both a prefix of $W_\s(Y)$ and a suffix of $W_\s(X)$, hence it must be its own inverse.
Notice $\cW_*$ is also of odd size---its first letter is the first letter of
$W_\s(Y)$ and its last letter is the last letter of $W_\s(X)$, so they have to either both be $a^{\pm 1}$ or $b^{\pm 1}$.
But that means $\cW_*$ cannot be its own inverse, since the letter in the exact middle of $\cW_*$ would have to be its own inverse, and that's not possible.
That contradiction concludes the proof of part 2.

\end{proof}

\begin{lemma} \label{lifting}
Let $\s$ be a periodic labelling, and $G_\s$ the corresponding subgroup of $\mathrm{Homeo}^+(\bR)$.
Then for some $n \in \bZ_+$ there exists a short exact sequence
\[
1 \to \bZ \to G_\s \to T(n) \to 1
\]
\end{lemma}
\begin{proof}

Let $n$ be the period of $\s$. Let $F(n), S(n) < T(n)$ be as in the statement of Lemma \ref{tgenerated}.
Next, define $\widetilde{F}(n) < \mathrm{Homeo}^+(\bR)$ to be the group of homeomorphisms $f : \bR \to \bR$, which satisfy
\begin{enumerate}
\item
$f$ is a piecewise linear homeomorphism of $\bR$ with a discrete set of breakpoints, all of which lie in $\bZ[\frac{1}{2}]$.
\item
$f'(x)$, wherever it exists, is an integer power of 2.
\item
For all $x, y \in \bR$, if $x - y = kn$ for some $k \in \bZ$, then $x \cdot f - x = y \cdot f - y$.
\item
$0 \cdot f = 0$.
\end{enumerate}
Notice $\widetilde{F}(n)$ is isomorphic to $F(n)$.
Lastly, define $\widetilde{S} < \mathrm{Homeo}^+(\bR)$ to be the group of all translations of $\bR$ by dyadic rationals.

First we show that $G_\s \leq \< \widetilde{F}(n), \widetilde{S} \>$.
Observe that for any $x, y \in \bR$, if $x - y = kn$ for some $k \in \bZ$, then
for all $g \in G_\s$ we have $x \cdot g - x = y \cdot g - y$. (This can be seen directly from the definition of $G_\s$,
or as a consequence of Lemma \ref{glocal} by taking $\t = \s$ and $w$ such that $g = w(S_\s)$.)
It follows that $g = fs$, where $s \in \widetilde{S}$ is a translation $s \colon x \mapsto x + 0 \cdot g$,
and $f$ is some element of $\widetilde{F}(n)$. Hence $G_\s \leq \< \widetilde{F}(n), \widetilde{S} \>$.

We will show that $\< \widetilde{F}(n), \widetilde{S} \>\leq G_{\s}$, and conclude that $G_\s = \< \widetilde{F}(n), \widetilde{S} \>$.
Since dyadic translations trivially satisfy the requirements of Theorem~\ref{gbig}, we immediately have $\widetilde{S} < G_\s$.
We need to show that $\widetilde{F}(n) < G_\s$ as well. All elements of $\widetilde{F}(n)$ satisfy conditions
(1), (2), and (4) from Theorem~\ref{gbig} directly from the definition of $\widetilde{F}(n)$.
We claim that to satisfy condition (3), it's enough to take $k_f = n$ for any $f$.
Take any $x, y \in \bR$ such that $x - y \in \bZ$. Since $\cW_\s(x, k_f)$ and $\cW_\s(y, k_f)$ have length $2n + 1$,
part 2 of Lemma \ref{periodicsubwords} tells us that $\cW_\s(x, k_f)$ is never equal to $\cW_\s^{-1}(y, k_f)$, so condition
(3.b) is satisfied by the virtue of its hypothesis always being false.
For condition (3.a), we have from part 1 of Lemma \ref{periodicsubwords}
that $\cW_\s(x, k_f)$ can equal $\cW_\s(y, k_f)$ only if $x - y$ is a multiple of $n$. But then $x - x \cdot f = y - y \cdot f$ holds for all $f \in \widetilde{F}(n)$
by the definition of $\widetilde{F}(n)$.

Having established that $G_\s = \< \widetilde{F}(n), \widetilde{S} \>$, we can now define the homomorphisms that make up
the desired short exact sequence. 
Let $\ph \colon \bZ \to G_\s$ map $k$ to a translation by $kn$.
%\[
%\ph \colon \bZ \ni k \mapsto \left( \bR \ni x \mapsto x + kn \right)
%\]
Note that $\ph(\bZ)$ is a central (and hence normal) subgroup in $G_{\s}$, and hence induces a natural quotient
$\psi \colon G_\s \to \textup{Homeo}^+(C(n))$. 

%Recall that we identify $C(n)$ with $\mathbf{R}/n\mathbf{Z}=[0, n)$ in the obvious way and using the notation $\phi_n:\mathbf{R}\to \mathbf{R}/n\mathbf{Z}$. One easily checks that $\ph$ and $\psi$ are indeed homomorphisms,

We claim that $\psi (G_\s)= T(n)$.
It is clear that $\ph$ is injective, $\ph(\bZ) = \ker(\psi)$, and that $\psi (G_\s)\leq T(n)$ by definition. To see that $\psi$ is surjective,
take any $t \in T(n)$. By Lemma \ref{tgenerated}, we can write $t = sf$ for some $f \in F(n)$ and $s \in S(n)$.
Let $\widetilde{f} \in \widetilde{F}(n)$ be the unique element of $\widetilde{F}(n)$ which agrees with $f$ on $[0, n)$,
and let $\widetilde{s} \in \widetilde{S}$ be a lift of the dyadic translation $s$. 
%Let $\a \in \bZ[\frac{1}{2}]$ be such that $s \colon x \mapsto x + \a \mmod n$.
%Then let $\widetilde{s} \colon x \mapsto x + \a$.
Since $G_\s = \< \widetilde{F}(n), \widetilde{S} \>$, we can define
$g \in G_\s$ by $g = \widetilde{s} \widetilde{f}$. Then one verifies that $\psi (g)=t$.
%maps $x$ to
%$x \cdot \widetilde{s} \widetilde{f} \mmod n$, and $x \cdot \widetilde{s} \widetilde{f} \mmod n = ( x \cdot \widetilde{s} ) \cdot \widetilde{f} \mmod n
%= ( x \cdot \widetilde{s} \mmod n ) \cdot f = (x + \a \mmod n) \cdot f = (x \cdot s) \cdot f = x \cdot fs = x \cdot t$,
%so $\psi g = t$.
\end{proof}

\section{Proof that $G_\rho$ does not have Kazhdan's property (T)}

The method of the proof is the same as in the preceding section. The key fact we will need is a theorem of Yehuda Shalom (see Theorem 6.7 in \cite{shalom2000rigidity}).
\begin{theorem}[Shalom] \label{kazhdanopen}
The subset of groups with Kazhdan's property (T) in the space of marked groups $\cG_m$ is open.
\end{theorem}
In light of Lemma \ref{limitofperiodic}, it is then sufficient if we prove that the groups $G_\s$ don't have property (T)
when $\s$ is periodic.
That will follow from a theorem by Matte Bon, Triestino and the second author (see Theorem 4.5 in \cite{LMBT}), which was also independently proved by Cornulier in \cite{Cornulier}.
\begin{definition}
A homeomorphism $h \colon \bS^1 \to \bS^1$ is {\em piecewise linear} if for all but finitely many points $x \in \bS^1$
there is a neighborhood $I(x)$ such that the restriction $\restr{h}{I(x)}$ is of the form $y \mapsto a y + b$.
The group of all piecewise linear homeomorphisms of $\bS^1$ is called $\PL(\bS^1)$.
\end{definition}
\begin{theorem}[Lodha, Matte Bon and Triestino \cite{LMBT}, Cornulier \cite{Cornulier}] \label{PLcircleNotKazhdan}
If $G$ is a countable Kazhdan group, every homomorphism $\rho \colon G \to \PL(\bS^1)$ has finite image.
\end{theorem}
\begin{proof}[Proof of Theorem \ref{notkazhdan}]
For every $n \in \bZ_+$, $T(n)$ is isomorphic to a subgroup of $\PL(\bS^1)$. Lemma \ref{limitofperiodic} then guarantees
that for any periodic labelling $\s$, there exists a homomorphism from $G_\s$ onto an infinite subgroup of $\PL(\bS^1)$,
which by Theorem \ref{PLcircleNotKazhdan} implies that $G_\s$ does not have property (T). By Lemma \ref{limitofperiodic}, it follows that $G_\rho$
is a limit of groups without property (T), and hence by Theorem \ref{kazhdanopen} we conclude that $G_\rho$ does not have property (T).
\end{proof}

\section{Proof that $G_\rho$ is not inner amenable}
Throughout this section we assume that the labelling $\rho$ is quasi-periodic.
We aim to prove theorem \ref{notinner} using the following criterion by Haagerup and Olesen (see Corollary 3.3 in \cite{haagerup2017non}):
\begin{proposition}[Haagerup and Olesen] \label{criterion}
Let $G$ be a discrete group. If $G$ has a non-amenable subgroup $H \leq G$ such that $\{ g \in H: gh = hg \}$ is amenable for all $h \in G \setminus \{ e \}$,
then $G$ is not inner amenable.
\end{proposition}

Let $H < \mathrm{Homeo}^+([0, 1])$ and $\lambda, \pi: H \to G_\rho$ be defined as in the preliminaries.
%If $E = [0, 1]$ or $E = \bR$, and $g \in \mathrm{Homeo}^+(E)$, we denote
%\[
%\Supp(g) = \{ x \in E : x \cdot g \neq x \}
%\]
%Let $F$ be the Thompson's group $F$, and $F' = [F, F]$. 
Choose $f \in F' = H'$ such that $\Supp(f) = \left(\frac{1}{16}, \frac{15}{16}\right)$,
$x \cdot f > x$ for
all $x \in (\frac{1}{16}, \frac{15}{16})$, and $\frac{2}{16} \cdot f > \frac{14}{16}$. Then proceeding as in Proposition 4.1 in \cite{HydeLodha},
$K = \< \lambda(f), \pi(f) \>$ is a subgroup of $G_\rho$ and is a free group, freely generated by $\lambda(f)$ and $\pi(f)$. Clearly $K$ is not amenable.

In the following we will use $\S$ to denote the set of generators of $K$ and their inverses:
\[
\S = \left\{ \lambda(f), \pi(f), \left( \lambda(f) \right)^{-1}, \left( \pi(f) \right)^{-1}  \right\}
\]

Recall that $x_0 \in \bR$ is called a transition point of $g \in G_\rho$ if $x_0 \in \cl{\Supp(g)} \setminus \Supp(g)$.
If $x_0$ is a transition point of $g$, then in particular $x_0 \cdot g = x_0$.
The following Lemma follows immediately from the definition of $G_{\rho}$ and from Lemma $5.1$ in \cite{HydeLodha}.
\begin{lemma}\label{transition}
Let $h, g \in G_\rho$ be non-identity elements. The following holds:
\begin{enumerate}
\item $Supp(g), Supp(h)$ have infinitely many connected components, each of which is a bounded open interval.
\item The set of transition points of $g,h$ are infinite discrete sets with no accumulation point in 
$\mathbf{R}$.
\item If $P$ is the set of transition points of $h$ and $hg=gh$, then $P\cdot g=P$.
\end{enumerate}
\end{lemma}

\begin{lemma} \label{notcommute}
Let $h, g \in G_\rho$. If $h$ has a transition point $x_0$ such that $x_0 \in \Supp(g)$, then $hg \neq gh$. In particular, if $h,g$ commute, then $g$ fixes every transition point of $h$. 
\end{lemma}
\begin{proof}
Let $P$ be the set of transition points of $h$.
Since $h,g$ commute, then from Lemma \ref{transition} it follows that $P\cdot g=P$.
However, since also from Lemma \ref{transition} every connected component of support of every element of $G_{\rho}$ is a bounded interval, it follows that any connected component of support of $g$ containing an element of $P$ must contain infinitely many elements of $P$, which contradicts the second part of Lemma \ref{transition}.
\end{proof}

From Proposition \ref{Fprimefff}, we shall derive the following key Corollaries.

\begin{cor}\label{nonfree}
Let $f_1,f_2\in G_{\rho}$ be elements that both fix a point $x\in \mathbf{R}$.
Then the group generated by $f_1,f_2$ is not the free group of rank $2$.
\end{cor}

\begin{proof}
Assume that $\langle f_1,f_2\rangle$ is a free group of rank $2$. Then the group generated by 
$$\alpha_1=[f_1,f_2]\qquad \alpha_2=[f_1^2,f_2^2]$$ is also free of rank $2$. However, since the group of germs of $f_1,f_2$ at the fixed point $x$ is abelian (since the slopes are powers of $2$), $\alpha_1,\alpha_2$ have trivial germs at $x$. It follows that $\alpha_1,\alpha_2$ pointwise fix a neighborhood of $x$, and hence satisfy the hypothesis of Proposition \ref{Fprimefff}. It follows that the group generated by $\alpha_1,\alpha_2$ is not the free abelian group of rank $2$, since a finite direct sum of copies of $F'$ does not contain nonabelian free groups. This contradicts our hypothesis.
\end{proof}

\begin{cor} \label{cyclicstabilizer}
For any $x_0 \in \bR$, the stabilizer of $x_0$ in $K$ is either trivial or isomorphic to $\bZ$.
\end{cor}
\begin{proof}
Since $K$ is a free group, the stabilizer of $x_0$ in $K$ is either trivial, isomorphic to $\bZ$, or a nonabelian free group. By Corollary \ref{nonfree}, the stabilizer of $x_0$ in $G_{\rho}$ (an overgroup of the stabilizer of $x_0$ in $K$) does not contain nonabelian free subgroups. It follows that the stabilizer of $x_0$ in $K$ cannot be a nonabelian free group.
\end{proof}

\begin{proof}[Proof of Theorem \ref{notinner}]
Let $h \in G_\rho \setminus \{ e \}$. %By Lemma 5.1 in \cite{HydeLodha}, $x \cdot h = x$ for some $x \in \bR$.
From Lemma \ref{transition} it follows that $h$ has an infinite set of transition points. We pick one and call it $x_0$.
By Lemma~\ref{notcommute}, the subgroup $C = \{ g \in K: gh = hg \}$ is a subgroup of the stabilizer
of $x_0$ in $K$, which by Lemma~\ref{cyclicstabilizer} is amenable, hence $C$ is also amenable.
At the same time, $K$ is not amenable, so we can apply Proposition \ref{criterion} to conclude that $G_\rho$ is not inner amenable.
\end{proof}

\section{Proof that $\b_1^{(2)}(G_\rho) = 0$}

We shall use the approach described by Arnaud Brothier in Chapter 5 of \cite{brothier2022forest}.
\begin{definition}
For a group $G$, we define a {\em good list of generators $L$} to be either a nonempty finite list
$(g_1, \ldots, g_k)$ with $k \geq 2$ or an infinite list $\{ g_j \}_{j = 1}^\i$ of elements of $G$ satisfying
\begin{itemize}
\item
the elements of $L$ generate $G$;
\item
two consecutive elements of $L$ commute;
\item
each element of the list (except possibly the last one) is required to have infinite order.
\end{itemize}
\end{definition}
\begin{proposition}[Brothier] \label{goodlistl2betti}
If $G$ admits a good list of generators, then its first $\ell^2$-Betti number is equal to 0.
\end{proposition}

\begin{proof}[Proof of Theorem \ref{l2betti}]
Recall that the subgroups $\cK$ and $\cL$ of $G_\rho$ were defined as
$\cK = \< \z_1, \z_2, \z_3 \>$ and $\cL = \< \chi_1, \chi_2, \chi_3 \>$,
and $G_\rho$ is generated by their commutator subgroups, $G_\rho = \< \cK' \cup \cL' \>$.
Below, we will construct a sequence containing all elements of $\cK' \cup \cL'$
such that every two consecutive elements of the sequence will commute.

Both $\cK$ and $\cL$ are isomorphic to $H = \< F', \n \>$. The isomorphisms are given by
$\lambda \colon H \to \cK$ and $\pi \colon H \to \cL$ defined in Section \ref{Prel}.
Furthermore, the commutator subgroup $H' < H$ is equal to $F'$, in other words,
consists of exactly those elements of $F$ whose support is included in some
subinterval $[a, b] \subseteq (0, 1)$. If $f \in H$ satisfies $\Supp(f) \subseteq [\e, 1 - \e] \subseteq (0, 1)$,
then $\Supp(\lambda(f)) \subseteq \bigcup_{k \in \bZ} [k + \e, k + 1 - \e]$
and $\Supp(\pi(f)) \subseteq \bigcup_{k \in \bZ} \left[k + \frac{1}{2} + \e, k + \frac{3}{2} - \e \right]$.

Based on the above, we know that for every element $f \in \cK'$ there exists $\e \in \left( 0, \frac{1}{4} \right)$ such that
$\Supp(f) \subseteq {\bR \setminus \bigcup_{k \in \bZ} [k - \e, k + \e]}$. We can then
find $f' \in \cK'$ whose support is included in ${\bigcup_{k \in \bZ} (k - \e, k + \e)}$ and hence $ff' = f' f$.
Similarly, for every $g \in \cL'$
there exists $\d \in \left( 0, \frac{1}{4} \right)$ such that $\Supp(g) \subseteq {\bR \setminus \bigcup_{k \in \bZ} \left[ k + \frac{1}{2} - \d, k + \frac{1}{2} + \d \right]}$,
and we can find $g' \in \cL'$ with
$\Supp(g') \subseteq \bigcup_{k \in \bZ} \left( k + \frac{1}{2} - \d, k + \frac{1}{2} + \d \right)$
and $gg' = g'g$. Notice that since $\e, \d < \frac{1}{4}$, we also have $f'g' = g'f'$.

Let $\{ f_n \}_{n = 0}^\i$ be a sequence containing all elements of $\cK'$, and let $\{ g_n \}_{n = 0}^\i$ contain all elements of $\cL'$.
We now define a new sequence $\{ h_n \}$ as follows. Let $h_0 = f_0$, $h_3 = g_0$, and let $h_1$ and $h_2$ be two elements
such that $[h_0, h_1] = [h_1, h_2] = [h_2, h_3] = e$. Then let $h_6 = f_1$, and $h_4, h_5$ be such that
$[h_3, h_4] = [h_4, h_5] = [h_5, h_6] = e$. And so on in that fashion. The sequence $\{ h_n \}$ contains all elements
of $\cK' \cup \cL'$ and each two consecutive elements of $\{ h_n \}$ commute.
Also, since $G_\rho$ is torsion free, all elements of $\{ h_n \}$ have infinite order.
So $\{h_n\}$ is a good list of generators for $G_\rho$, and it follows by Proposition \ref{goodlistl2betti}
that $\b_1^{(2)}(G_\rho) = 0$ as promised.

\end{proof}

\section{Uniform simplicity}

The goal of this section is to provide the proof of Theorem \ref{UniformlySimpleMainTheorem}.
Throughout this section we fix a quasi-periodic labelling $\rho$.
Recall that given an element $f$ in a group, we denote by $C_f$ the conjugacy class of $f$ in the group.
A group $G$ is said to be $n$-uniformly simple if for each $\alpha,\beta\in G\setminus \{e\}$, $\alpha$ is a product of at most $n$ elements in $C_{\beta}\cup C_{\beta^{-1}}$.
We shall need the following fact.
\begin{theorem}[Theorem $1.1$ in \cite{prox}]\label{USF}
The derived subgroup of Thompson's group $F$ is $6$-uniformly simple.
\end{theorem}
%Moreover, the following holds. 
Let $G_1,...,G_m$ be groups. An element of a finite direct sum $\bigoplus_{1\leq i\leq m}G_i$ is called \emph{full}, if its image under the projection onto each coordinate is nontrivial.

\begin{lem}\label{directsum}
Denote by $F'$ the derived subgroup of Thompson's group $F$.
There is a $k\in \mathbf{N}$ such that the following holds.
Let $\bigoplus_{1\leq i\leq n} F_i'$ be a direct sum, where each $F_i'\cong F'$ and $n\in \mathbf{N}$ is arbitrary.
Then for each pair of full elements $f,g\in \bigoplus_{1\leq i\leq n} F_i'$, $g$ is a product of at most $k$ elements in $C_{f}\cup C_{f^{-1}}$. 
\end{lem}

\begin{proof}
We know from Theorem \ref{USF} that for every pair of nontrivial elements $\alpha,\beta\in F_i'$,
$\alpha$ can be expressed as a product of at most six elements in $C_{\beta}\cup C_{\beta^{-1}}$.

Now let $$f=(f_i)_{1\leq i\leq n}\qquad g=(g_i)_{1\leq i\leq n}\qquad f_i,g_i\in F_i'$$
We choose $h_1,...,h_n\in F'$ such that $l_i=[h_i,f_i]\neq e$.
It follows that $g_i$ is also a product $$g_i=\gamma_1^{(i)}...\gamma_6^{(i)}\qquad \gamma_j^{(i)}\in C_{l_i}\cup C_{l_i^{-1}}\cup \{1\}$$
For the index $i$, this provides a finite sequence which is an element of $\{+,-,\emptyset\}^{<\mathbf{N}}$, as follows.
The $j$'the term of the sequence is $+$ if $\gamma_j^{(i)}\in C_{l_i}$, $-$ if $\gamma_j^{(i)}\in C_{l_i^{-1}}$ and the string $\emptyset$ if $\gamma_j^{(i)}=1$. %(We use the convention that the empty strings only occur at the end of the sequence, if any).
We call this the \emph{signature of the index $i$}.
Clearly, there are fewer than $3^6$ such possible signatures.
Whenever two indices $i,j$ have the same signature, the element $g_i\oplus g_j\in F'\oplus F'$ can be expressed as a product of at most six elements in $C_{l_i\oplus l_j}\cup C_{l_i^{-1}\oplus l_j^{-1}}$.

We partition $\{1,...,n\}$ into (at most $3^6$) sets, where each set consists of indices $i$ that supply the same signature.
For each such set $A\subseteq \{1,...,n\}$, the element $$\tau_A=\sum_{i\in A}[h_i,f_i]\oplus \sum_{i\in A^c} [id,f_i]$$
has the following feature.
The element $$\sum_{i\in A}g_i\oplus \sum_{i\in A^c} e$$ can be expressed as a product of at most six elements in $C_{\tau_A}\cup C_{\tau_A^{-1}}$,
hence a product of at most $12$ elements in $C_{f}\cup C_{f^{-1}}$.
Applying the same argument to each of the (at most) $3^6$ sets, we obtain the desired result for $k=3^6\times 12$.
\end{proof}

The following is a key proposition in our proof of uniform simplicity.

\begin{prop}\label{mainpropUS}
Let $f\in G_{\rho}$ be an element that is uniformly stable.
Then $f$ lies in a subgroup $A$ of $G_{\rho}$ that is isomorphic to a finite direct sum of copies of $F'$,
and $f$ is full in $A$.
\end{prop}

We shall need the following, which is an immediate consequence of quasi-periodicity and the definition of the global definition of the 
groups $G_{\rho}$ in Definition \ref{Krho}.

\begin{lem}\label{uniformlystable}
Let $\rho$ be a quasi-periodic labelling.
Let $f\in G_{\rho}$ be an element that fixes a neighborhood of $0$ pointwise.
Then $f$ is uniformly stable.
\end{lem}

%Now we are ready to recall the proof of Proposition \ref{mainpropUS}.
\begin{proof}[Proof of Proposition \ref{mainpropUS}]
We shall provide an explicit description of the group $A$.
Let $f\in G_{\rho}$ be uniformly stable.
Let $\{I_{\beta}\mid \beta\in P\}$ be the set of atoms of $f$.
Recall the definition of the constant $k_f\in \mathbf{N}$ from Definition \ref{Krho}.
(Since Definition \ref{Krho} provides a characterisation of elements of $G_{\rho}$,
it supplies such a constant for each element of $G_{\rho}$).
We denote by $l_f=k_f+l$, where $l=\max\{|I_{\beta}|\mid \beta\in P\}$. 
Note that since $f$ is uniformly stable, the number of decorated atoms $\mathcal{T}(f)$ is finite for any $n\in \mathbf{N}$,
so this quantity is defined.

Let the cellular decomposition of $f$ as decorated atoms $\mathcal{T}_{l_f}(f)$ be $f_{\zeta_1},...,f_{\zeta_m}$.
Here we represent the equivalence classes of decorated atoms in $\mathcal{T}_{l_f}(f)$ as $\zeta_1,...,\zeta_m$.
For each $1\leq i\leq m$, set $L_i=|I_{\alpha}|$ where $(I_{\alpha},l_f)\in \zeta_i$.
(Recall that $|I_{\alpha}|=|I_{\beta}|$ whenever $(I_{\alpha},l_f),(I_{\beta},l_f)\in \zeta_i$.)
For each $1\leq i\leq m$, define the canonical isomorphism $$\phi_i:F'\to F_{[0,L_i]}'$$
where $F_{[0,L_i]}$ is the standard copy of $F$ supported on the interval $[0,L_i]$.

For each $1\leq i\leq m$, we have $$\{\mathcal{W}(I_{\alpha},l_f)\mid (I_{\alpha},l_f)\in \zeta_i\}=\{W_i,W^{-1}_i\}$$
for some words $W_1,...,W_m$.
Define a map $$\phi: \bigoplus_{1\leq i\leq m}F'\to \textup{Homeo}^+(\mathbf{R})$$ as follows. 
For $\alpha\in P$ and $1\leq i\leq m$: $$\phi(g_1,...,g_m)\restriction I_{\alpha}\cong_T \phi_i(g_i)\qquad \text{ if }(I_{\alpha},l_f)\in \zeta_i\text{ and }\mathcal{W}(I_{\alpha},l_f)=W_i$$
$$\phi(g_1,...,g_m)\restriction I_{\alpha}\cong_T \iota_{L_i}\circ\phi_i(g_i)\circ \iota_{L_i}\qquad \text{ if }(I_{\alpha},l_f)\in \zeta_i\text{ and }\mathcal{W}(I_{\alpha},l_f)=W^{-1}_i$$
where $\iota_{L_i}:[0,L_i]\to [0,L_i]$ is the unique orientation reversing isometry. 
It is easy to check that this is an injective group homomorphism, since the image of each element satisfies Definition \ref{Krho} with a uniform constant $l_f$.
It is easy to see that the image of $\phi$ contains $f$, which is equal to $\phi (\phi_1^{-1}(f_{\zeta_1}), \ldots, \phi_m^{-1}(f_{\zeta_m}))$,
and hence is full in $A$.
\end{proof}

We shall also need the following Proposition which is a small variation of Proposition \ref{Fprimefff}, which follows from the above.

\begin{prop}\label{Fprime}
Let $\rho$ be a quasi-periodic labelling.
Let $f,g\in G_{\rho}$ be elements with the following property. There is a compact interval $I$ such that both $f,g$ fix each point in $I$.
%$$x\cdot f=x\qquad x\cdot g=x\qquad \forall x\in I$$
Then there is a subgroup $A<G_{\rho}$ with the following properties.
\begin{enumerate}
\item $A$ is isomorphic to a finite direct sum of copies of $F'$.
\item $f,g\in A$ and are both full in $A$.
\end{enumerate}
\end{prop}

\begin{proof}
%(See the second half of the proof of Lemma $2.2$, for instance.)
Using the minimality of the action of $G_{\rho}$, we can find an element $h\in G_{\rho}$ such that $h^{-1} fh,h^{-1}gh$ pointwise fix a neighborhood of $0$.
It follows from Lemma \ref{uniformlystable} that these elements are uniformly stable.
Finally, the proof of Proposition \ref{mainpropUS} above applies to $h^{-1} fh,h^{-1}gh$.
Indeed, they live in the subgroups $A_1,A_2$ of $G_{\rho}$ that are supplied by this proof.
We can perturb the constructions of these subgroups by considering a common refinement of the cellular decompositions of both elements,
in which the atoms are allowed to have integer points whose neighborhoods are pointwise fixed,
to obtain a subgroup $A$ with the desired properties.
\end{proof}

\begin{proof}[Proof of Theorem \ref{UniformlySimpleMainTheorem}]

We will show that there is an $m\in \mathbf{N}$ such that given $\alpha,\beta\in G_{\rho}\setminus \{e\}$, $\alpha$ is a product of at most $m$ elements in $C_{\beta}\cup C_{\beta^{-1}}$.
The proof shall be done in two steps:
\begin{enumerate}

\item[(Step 1)] %\marginpar{Isn't uniform stability invariant under conjugaction? Perhaps adding a footnote?}
We find elements $\nu_1,\nu_2,\nu_3,\nu_4\in G_{\rho}$ such that $\nu_1\nu_2=\alpha$
and the elements $\alpha_1= \nu_3^{-1}\nu_1\nu_3$ and $\alpha_2=\nu_4^{-1}\nu_2\nu_4$ pointwise fix a nonempty open neighborhood of $0$.
It follows from Proposition \ref{Fprime} that $\alpha_1,\alpha_2$ lie in subgroups $A_1,A_2$ of $G_{\rho}$ that are both isomorphic to a finite direct sum of copies of $F'$,
and $\alpha_1,\alpha_2$ are full in $A_1,A_2$ respectively.

\item[(Step 2)] We construct an element $\gamma\in G_{\rho}$ such that $\beta_1=\gamma^{-1} \beta\gamma$ satisfies that $V\cdot \beta_1\cap V=\emptyset$, where $V$ is the set 
\[
V=\bigcup_{n\in \tfrac{1}{2}\mathbf{Z}\setminus \mathbf{Z}}(n-\frac{1}{16},n+\frac{1}{16})
%\left\{x\in \mathbf{R} : |x-y|\leq \tfrac{1}{16}
%\text{ for some }y\in \tfrac{1}{2}\mathbf{Z}\setminus \mathbf{Z}
%\right\}
\]
%Then $\beta_1$ .
It follows that for any triple of nontrivial symmetric elements $f_1,f_2,f_3\in \mathcal{K}$ such that: $$\Supp(f_1),\Supp(f_2),\Supp(f_3)\subset V \text{ and } [f_1,f_2]=f_3,$$
we have that $[[f_1, \beta_1^{-1}],f_2]=f_3\in \mathcal{K}$.
We end by observing that any symmetric element of $\mathcal{K}$ lies in $A_1\cap A_2$, and moreover, is full in $A_1,A_2$ (the groups from Step $1$).
It follows that $f_3 \in A_1\cap A_2$ 
and $f_3$ is full in $A_1,A_2$.

\end{enumerate}

{\bf Conclusion}: 
Combining the above steps, and applying Lemma \ref{directsum}, 
we obtain that $\alpha_1,\alpha_2$ are products of at most $k$ elements in $C_{f_3}\cup C_{f_3^{-1}}$,
where $k$ is the constant from Lemma \ref{directsum}.  
Therefore, $\alpha_1,\alpha_2$ are products of at most $4k$ elements in $C_{\beta}\cup C_{\beta^{-1}}$.
We conclude that $\alpha$ is a product of at most $8k$ elements in $C_{\beta}\cup C_{\beta^{-1}}$.\\

{\bf Proof of step 1}: Recall that each element in $G_{\rho}$ admits fixed points in $\mathbf{R}$ (see part $(2)$ of Lemma $5.1$ in \cite{HydeLodha}).
Let $x\in \mathbf{R}$ be a fixed point of $\alpha$.
We assume without loss of generality that $x\in \bigcup_{n\in \mathbf{Z}} (n,n+1)$.
The case when $x\in \bigcup_{n\in \mathbf{Z}} (n-\frac{1}{2},n+\frac{1}{2})$ is similar and uses an element in the group $\mathcal{L}$ instead of $\mathcal{K}$ in the argument below.

We can find a symmetric element $f \in \mathcal{K}$ such that:
\begin{enumerate}
\item $\overline{\Supp(f)}\subset \bigcup_{n\in \mathbf{Z}} (n,n+1)$, and hence $f$ is uniformly stable.
\item $x\cdot f=x$ and the right derivatives of $f,\alpha$ at $x$ coincide.
\end{enumerate}
Let $\nu_1=f$ and $\nu_2=f^{-1} \alpha$.
Note that both $\nu_1,\nu_2$ are uniformly stable.
Let $x_1,x_2\in \mathbf{R}, \epsilon >0$ be such that for each $1\leq i\leq 2$ the element $\nu_i$ fixes the interval $(x_i-\epsilon, x_i+\epsilon)$ pointwise. 
%$$y\cdot \nu_i=y\qquad  \forall y\in (x_i-\epsilon, x_i+\epsilon)$$
Using minimality of the action of $G_{\rho}$ on $\mathbf{R}$, we find elements $\nu_3,\nu_4\in G_{\rho}$ that satisfy 
$$0\cdot \nu_3^{-1} \in (x_1-\epsilon, x_1+\epsilon)\qquad 0\cdot \nu_4^{-1} \in (x_2-\epsilon, x_2+\epsilon)$$
It follows that $\nu_1,\nu_2,\nu_3,\nu_4$ are the required elements.\\

{\bf Proof of step 2}:
We shall need the following Lemma.

\begin{lem}\label{Lem1Step2}
There is an $\epsilon>0$, $p\in \mathbf{N}$ and an infinite discrete set $X\subset \mathbf{R}$ 
such that for each $x\in X$ and $U_x=(x-\epsilon,x+\epsilon)$, we have:
\begin{enumerate}
\item $U_x\cup U_x\cdot \beta\subset [n,n+1]$ for some $n\in \mathbf{Z}$ and $U_x\cap U_x\cdot \beta=\emptyset$.
\item For each interval $[n,n+p]$ for $n\in p\mathbf{Z}$, $[n,n+p]\cap X\neq \emptyset$.
\item For any pair $n_1,n_2\in p\mathbf{Z}$,
if $$\mathcal{W}([n_1,n_1+p],1)=\mathcal{W}([n_2,n_2+p],1)$$
then there are $x_1,x_2\in X$ such that $U_{x_i}\subset [n_i,n_i+p]$ and
$T(U_{x_1})=U_{x_2}$
where $T:[n_1,n_1+p]\to [n_2,n_2+p]$ is the orientation preserving isometry.
\end{enumerate}
\end{lem}

\begin{proof}
Recall that there is a $p_1\in \mathbf{N}$ such that $\beta$ admits a fixed point $y$, which is also a transition point, in any compact interval of length at least $p_1$.
(This follows from the proof of part $(2)$ of Lemma $5.1$ in \cite{HydeLodha}, and also directly from the definition of the group $G_{\rho}$ using 
the fact that $\rho$ is quasi-periodic).
Next, recall the constant $k_{\beta}$ from Definition \ref{Kset}.
We choose $p=\textup{max}\{p_1,k_{\beta}\}$.
 
Let $n\in \mathbf{Z}$ be such that there is a transition point $y$ of $\beta$ that lies in $[n,n+1]$ and has a nontrivial germ in $[n,n+1]$.
Then since $\beta$ is a homeomorphism, we can find suitable $x$ and $\epsilon$ such that $(x-\epsilon,x+\epsilon)$ has the property that: $$(x-\epsilon,x+\epsilon)\cdot \beta \cap (x-\epsilon,x+\epsilon)=\emptyset\qquad (x-\epsilon,x+\epsilon)\cdot \beta \cup (x-\epsilon,x+\epsilon)\subset [n,n+1]$$
%If $y=n$ and $\beta$ has a trivial germ in $[n,n+1]$ at $y$, we simply choose a suitable $(x-\epsilon,x+\epsilon)\subset [n-1,n]$.

So far we have chosen $x,\epsilon$ for each such transition point.
Our goal now is to make the choices of such points $x$ satisfy condition $(3)$ and the choice of $\epsilon$ uniform over all such $x$.
The former is a straightforward application of Definition \ref{Kset}. The choice of $\epsilon$ above can be made uniform over all such transition points $y$, as follows. 
Using quasi-periodicity of the labelling, one observes the following. For an element of $G_{\rho}$,
the set of slopes, wherever they exist, is a finite subset of $\{2^n : n\in \mathbf{Z}\}$.
This makes it possible for us to choose such an $\epsilon$.
The collection of such points $x$ is then the set $X$, and has the required properties.
\end{proof}

Now we shall proceed to finish the proof of Step $2$.
Let $U=\bigcup_{x\in X}U_x=\bigcup_{x\in X}(x-\epsilon,x+\epsilon)$ be the set from the previous Lemma.
Recall that 
\[
V=\bigcup_{n\in \tfrac{1}{2}\mathbf{Z}\setminus \mathbf{Z}}(n-\frac{1}{16},n+\frac{1}{16})
%V=\left\{x\in \mathbf{R} : |x-y|\leq \tfrac{1}{16}
%\text{ for some }y\in \tfrac{1}{2}\mathbf{Z}\setminus \mathbf{Z}
%\right\}
\]
Our goal is to produce an element $\gamma\in G_{\rho}$ such that $V\cdot \gamma^{-1}\subset U$.
Note that this implies that $$V\cdot (\gamma^{-1} \beta \gamma) \cap V=\emptyset$$
thereby finishing the proof of step $2$.

For each pair consisting of a compact interval $[n,n+p]$ and the word ${\mathcal{W}([n,n+p],1)}$,
we choose a special element supplied by Proposition \ref{SpecialElements2} whose inverse maps $V\cap [n,n+p]$ inside $U\cap [n,n+p]$.
A product of finitely many such special elements is then the required $\gamma$.
Finally, the fact that $\mathcal{K}< A_1\cap A_2$, and is full in both $A_1,A_2$, follows from the explicit definitions of $A_1,A_2$ supplied by the proof of Proposition \ref{mainpropUS}.

%\begin{lem}
%Let $X$ and $\{U_x\mid x\in X\}$ be as in Lemma \ref{Lem1Step2}.
%There is an $l\in \mathbf{N}$ such that:
%\begin{enumerate}
%\item For each $n\in l\mathbf{Z}$, there is at least one $x\in X$ such that $U_x\subset [n,n+l]$.
%\item For any pair $n_1,n_2\in l\mathbf{Z}$,
%if $$\mathcal{W}([n_1,n_1+l],1)=\mathcal{W}([n_2,n_2+l],1)$$
%there there are $x_1,x_2\in X$ such that $U_{x_i}\subset [n_i,n_i+l]$ for $1\leq i\leq 2$ and
%$$T_{J_1,J_2}(U_{x_1})=U_{x_2}\qquad J_i=[n_i,n_i+l], 1\leq i\leq 2$$
%\end{enumerate}
%\end{lem}

%\begin{proof}
%Recall that from part $(3)$ of Definition \ref{Krho} it holds that there is a number $k_{\beta}$ such that 
%\begin{enumerate}
%\item Whenever $v,y\in \mathbf{R}$ satisfy that $$v-y\in \mathbf{Z}\qquad \mathcal{W}(v,k_{\beta})=\mathcal{W}(y,k_{\beta})$$ it holds that $$v-v\cdot \beta = y-y\cdot \beta$$
%\item Whenever $v,y\in \mathbf{R}$ satisfy that $$v-y\in \mathbf{Z}\qquad \mathcal{W}(v,k_{\beta})=\mathcal{W}^{-1}(y,k_{\beta})$$ it holds that $$v-v\cdot \beta=y'\cdot \beta-y' \qquad \text{ where }y'=y\cdot \iota$$
%\end{enumerate}
%Note that if in the above we replace $k_{\beta}$ by a number $k>k_{\beta}$, the above also automatically holds.

%Moreover, by Lemma \ref{Lem1Step2}, the following holds. There is a number $p\in \mathbf{N}$ such that for each $n\in p\mathbf{Z}$, the interval $[n,n+p]$ contains some $U_x$.
%Choosing $l=sup\{p,k_{\beta}\}$, we are done.
%\end{proof}
%Now we finish the proof of step $2$.

\end{proof}

\section{Property \textup{(FA)}}

In this section we shall prove that the groups $G_{\rho}$ have Serre's property $\textup{(FA)}$.
That is, every action of $G_{\rho}$ on a simplicial tree has a fixed point.
The following is stated as Corollary $2$ on page $64$ in \cite{serretrees}.

\begin{prop}\label{TreeProp1}
Let $G$ be a group action on a simplicial tree by automorphisms.
Assume that $G$ is generated by a finite set of elements $s_1,...,s_m$ such that each $s_i$ and $s_is_j$ for $i,j\in \{1,...,m\}$ admit fixed points.
Then $G$ admits a global fixed point.
\end{prop}

Our goal in this section is to construct a finite generating set for $G_{\rho}$ that satisfies the conditions of the previous Proposition.
We shall need the following Lemma.

\begin{lem}\label{TreeLem1}
Let $G=\bigoplus_{1\leq i\leq n}F'$ for some $n\in \mathbf{N}$.
For each action of $G$ on a simplicial tree by automorphisms, each non trivial element of $G$ admits a fixed point.
\end{lem} 

\begin{proof}
Consider such an action of $G$ on a simplicial tree $T$.
Assume by way of contradiction that $f\in G\setminus \{e\}$ does not admit a fixed point on $T$.
Then $f$ acts on $T$ as a hyperbolic element with a unique translation axis $L$ which is a simplicial line.

Consider the standard dynamical realisation of $F'<\textup{Homeo}^+[0,1]$.
We view the direct sum $G$ as $G<\textup{Homeo}^+[0,n]$,
where the $i$'th summand acts on an interval $[i-1,i]$ for $0< i\leq n, i\in \mathbf{N}$ in the standard fashion and fixes $[0,n]\setminus [i,i+1]$ pointwise. 

We find dyadic intervals $\{J_i\mid 1\leq i\leq n\}$ such that:
\begin{enumerate}
\item $J_i\subset (i-1,i)$.
\item $\Supp(f)\cap (\bigcup_{1\leq i\leq n} J_i)=\emptyset$
\end{enumerate}

Let $F_i'$ be the copy of $F'$ supported on $\interior(J_i)$.
Since each element of $F_i'$ commutes with $f$,
and since $F_i'$ is simple, it follows that $F_i'$ fixes the axis $L$ pointwise.
It follows that the group $\Xi=\bigoplus_{1\leq i\leq n}F_i'$, where each $F_i$ is supported on $J_i$, fixes the axis $L$ pointwise and hence comprises entirely of elliptic elements.

Using the transitivity of the action of $F'$, we can find an element $g\in G$ such that $$\Supp(g^{-1} f g)\subset \bigcup_{1\leq i\leq n}\interior(J_i)$$
In particular, it follows that $g^{-1} fg\in \Xi$ and hence it must fix a point in the tree.
This is a contradiction.
\end{proof}

{\em Proof of the second part of Theorem \ref{notkazhdan}}
Let $T$ be a simplicial tree upon which $G_{\rho}$ admits an action by simplicial automorphisms.
We shall find a generating set $S_{\rho}'$ for $G_{\rho}$ that satisfies the hypothesis of Proposition \ref{TreeProp1}.
%We recall Definition \ref{SimpleGroup}.
It is an elementary exercise to modify the generating set $S_{\rho}$ to a generating set $S_{\rho}'$ with the following property.
For all $f,g\in S_{\rho}'$, the elements $f$ and $g$ satisfy that the following.
There exists a compact interval $I$ with nonempty interior such that $$x\cdot f=x, \qquad x\cdot g=x \qquad \text{ for all } x\in I.$$ 
This can be done by replacing the generators of $H<F$ in Definition \ref{groupH} by generators with small support.

It follows from Proposition \ref{Fprime} that for each pair $f,g\in S_{\rho}'$, there is a subgroup $\Xi<G_{\rho}$
such that:
\begin{enumerate}
\item $\Xi$ is isomorphic to a finite direct sum of copies of $F'$.
\item $f,g\in \Xi$.
\end{enumerate}
From Lemma \ref{TreeLem1} it follows that $f,g,fg$ admit fixed points.
It follows that $S_{\rho}'$ is a generating set for $G_{\rho}$ that satisfies the hypothesis of Proposition \ref{TreeProp1}.
It follows that the action of $G_{\rho}$ admits a fixed point.

\bibliographystyle{amsalpha}
\bibliography{bib}

\end{document}